\documentclass[11pt]{amsart}
\usepackage{graphicx,color}
\usepackage{amsmath,amscd,amssymb,amsthm,amsfonts,multicol,ytableau,comment,lipsum}
\usepackage[vcentermath]{youngtab}
\usepackage{setspace}
\usepackage{CJK}
\usepackage[margin=1in]{geometry}
\usepackage[style=numeric, backend=bibtex]{biblatex}
\usepackage{tikz}
\usetikzlibrary{calc,backgrounds}
\usepackage{subcaption}

\theoremstyle{definition}
\newtheorem{theorem}{Theorem}
\newtheorem{corollary}[theorem]{Corollary}
\newtheorem{proposition}[theorem]{Proposition}
\newtheorem{lemma}[theorem]{Lemma}
\newtheorem{conjecture}[theorem]{Conjecture}
\newtheorem{definition}[theorem]{Definition}

\newtheorem{example}[theorem]{Example}
\newtheorem{remark}[theorem]{Remark}

\numberwithin{theorem}{section}
\numberwithin{equation}{section}

\newtheorem*{theoremA}{Theorem \ref{thm:A}}
\newtheorem*{theoremB}{Theorem \ref{thm:B}}
\newtheorem*{theoremC}{Theorem \ref{thm:C}}
\newtheorem*{theoremD}{Theorem \ref{thm:D}}
\newtheorem*{theoremE}{Theorem \ref{thm:E}}

\newcommand{\N}{\mathbb{N}}
\newcommand{\Z}{\mathbb{Z}}
\newcommand{\C}{\mathbb{C}}

\newcommand{\Sn}{\mathfrak{S}_n}

\newcommand{\Inv}{\mathrm{Inv}}
\newcommand{\inv}{\mathrm{inv}}
\newcommand{\PSPT}{\mathrm{PSPT}}
\newcommand{\SYT}{\mathrm{SYT}}
\newcommand{\inc}{\mathrm{inc}}
\newcommand{\asc}{\mathrm{asc}}
\newcommand{\xvars}{\mathbf{x}}
\newcommand{\Hess}{\mathrm{Hess}}

\title[Higher Specht bases and $q$-series for certain Hessenberg varieties]{Higher Specht bases and $q$-series for the cohomology rings of certain Hessenberg varieties}

\author{Kyle Salois}
\address{Department of Mathematics, Colorado State University, Fort Collins, CO, USA}
\email{kyle.salois@colostate.edu}
\thanks{The author was partially supported by NSF DMS award number 2054391.}
\date{\today}

\begin{document}

\begin{abstract}
It is conjectured (following the Stanley-Stembridge conjecture) that the cohomology rings of regular semisimple Hessenberg varieties yield permutation representations, but the decompositions of the modules are only known in some cases. For the Hessenberg function $h=(h(1),n,\ldots,n)$, the structure of the cohomology ring was determined by Abe, Horiguchi, and Masuda in 2017. We define two new bases for this cohomology ring, one of which is a higher Specht basis, and the other of which is a permutation basis.  We also examine the transpose Hessenberg variety, indexed by the Hessenberg function $h' = ((n-1)^{n-m},n^m)$, and show that analogous results hold. Further, we give combinatorial bijections between the monomials in the new basis and sets of $P$-tableaux, motivated by the work of Gasharov \cite{Gasharov}, illustrating the connections between the $\Sn$ action on these cohomology rings and the Schur expansion of chromatic symmetric functions.
\end{abstract} 

\maketitle{} 

\onehalfspacing

\section{Introduction}

Hessenberg varieties, originally defined in \cite{DePrSh} by De Mari, Procesi, and Shayman, are subvarieties of the full flag variety with deep connections to algebraic geometry, representation theory, and symmetric function theory. Given a Hessenberg function $h:[n]\to [n]$ and an $n\times n$ matrix $X$ over $\C$, the associated Hessenberg variety is $$ \Hess(X,h) := \{ V_{\bullet}\in \mathrm{Fl}(\C^n)\,|\,XV_j \subseteq V_{h(j)} \textrm{ for all } 1\leq j\leq n\}$$ where $\mathrm{Fl}(\C^n)$ is the flag variety of $\C^n$ consisting of sequences $V_{\bullet} = (V_1 \subset V_2 \subset \ldots \subset V_n = \C^n)$ of nested subspaces with $\mathrm{dim}(V_i) = i$. Different specializations of Hessenberg varieties include $\mathrm{Fl}(\C^n)$ itself, Springer fibers, and Peterson varieties. The study of Hessenberg varieties has yielded many results, including those found in \cite{AHM,AbHaHoMa,BC18,CHL20,GP16,HHMPT}. In particular, Hessenberg varieties are smooth when $X=S$ is a regular semisimple matrix, and the cohomology ring is known to be the quotient of a direct sum of polynomial rings. Further, in \cite{Ty07}, Tymoczko defines an action of $\Sn$ on the cohomology ring $H^{*}(\Hess(S,h))$ in this case. 

In \cite{StanCSF}, Stanley defined the chromatic symmetric function, which is a generalization of the chromatic polynomial of a graph, defined as follows: $$ X_{G}(\xvars) = \sum_{\kappa:V(G)\to \N} x_{\kappa(1)}\cdots x_{\kappa(n)}$$ where the sum is over proper colorings of the vertices of $G$, and the monomials keep track of how many times each color was used in the coloring. The Stanley-Stembridge conjecture \cite{StanCSF, StSt} states that for certain graphs $G$, the chromatic symmetric function $X_G$ can be expanded with nonnegative coefficients in the elementary symmetric function basis. In \cite{Gasharov}, Gasharov showed how to expand the chromatic symmetric function for these graphs in the Schur basis with nonnegative coefficients using $P$-tableaux. 

In \cite{SW}, Shareshian and Wachs defined a quasisymmetric generalization $X_G(\xvars;q)$ of $X_G(\xvars)$, proved that it is symmetric for the class of graphs studied in this problem, and extend Gasharov's result on $P$-tableaux to the quasisymmetric case. They further conjectured that the $\Sn$-representation of $H^*(\Hess(S,h))$ is related to the chromatic quasisymmetric function via the graded Frobenius map. This conjecture was proven in \cite{BC18} by Brosnan and Chow, and separately in \cite{GP16} by Guay-Paquet. 
From these sources, results on the cohomology rings of Hessenberg varieties can be translated to results on chromatic (quasi)symmetric functions, and likewise in reverse. In particular, showing that the $\Sn$ action on $H^*(\Hess(S,h))$ is a permutation representation would imply that the Stanley-Stembridge conjecture holds. More details on this problem can be found in \cite{SW}.

A refined understanding of the cohomology ring of regular semisimple Hessenberg varieties when $h=(h(1),n,\ldots, n)$ was given by Abe, Horiguchi, and Masuda in \cite{AHM} as the quotient of a single polynomial ring in two sets of variables. The $\Sn$-action on this cohomology ring is more accessible, and so connections can be drawn more easily between the cohomology ring and the corresponding symmetric functions. 

In this paper, we examine the presentation of $H^{*}(\Hess(S,h))$ for the Hessenberg functions $h=(h(1),n,\ldots,n)$ and $h'=((n-1)^{n-m},n^m)$, and draw connections between the monomial basis and $P$-tableaux. We are primarily motivated by Gasharov's expansion of $X_G(\xvars)$ into Schur functions using $P$-tableaux, and the generalization of this result by Shareshian and Wachs. Our first main result is the following theorem, which provides a higher Specht basis for the cohomology ring.

\begin{theorem}\label{thm:A}
The following sets form a basis of $H^{*}(\Hess(S,h))$ when $h=(h(1),n,\ldots,n)$: 
\begin{equation*}
    \begin{split}
    x_1^{i_1}x_2^{i_2}\cdots x_n^{i_n} \,\,\,\, &  \mbox{ which does not contain the factor } \prod_{\ell=1}^{h(1)} x_{\ell}  \\
    x_n^{\ell_1}x_{n-1}^{\ell_2}\cdots x_2^{\ell_{n-1}} (y_{k+1}-y_1) \,\,\,\, & \mbox{ which does not contain the factor } \prod_{\ell=h(1)+1}^{n} x_{\ell} 
    \end{split}
\end{equation*}
running over all $0\leq i_j \leq n-j$ in the first equation, and over all $0\leq \ell_j\leq n-1-j$, and $1\leq k\leq n-1$ in the second equation. We denote these sets of monomials $B_1$ and $B_3$. 
\end{theorem}

Higher Specht bases have the advantage of more clearly displaying the decomposition of the corresponding $\Sn$ representation into irreducibles, giving a deeper understanding of these cohomology rings. Since the basis elements of the irreducible Specht modules are in bijection with standard tableaux, we are motivated to find bijections between the monomials in Theorem \ref{thm:A} and certain sets of tableaux. These results are summarized in the following theorems. 

\begin{theorem}\label{thm:B}
    There exists a combinatorial, weight-preserving bijection between monomial basis elements of the first kind and the set of $P_h$-tableaux of shape $(1^n)$. 
\end{theorem}

\begin{theorem}\label{thm:C}
    There exists a combinatorial, weight-preserving bijection between monomial basis elements of the second kind with the set of pairs $(S,T)$ where $S$ is a standard tableau, and $T$ is a $P_h$-tableau, both of shape $(2,1^{n-2})$.
\end{theorem}

Further, in the case of regular nilpotent Hessenberg varieties, the cohomology ring is isomorphic to the $\Sn$-fixed points of the regular semisimple case. Let $N$ be a regular nilpotent matrix. We display a bijection between known basis elements of $H^{*}(\Hess(N,h))$ and the set of $P_h$-tableaux of shape $(1^n)$ for any Hessenberg function $h$.

\begin{theorem}\label{thm:D}
    There exists a combinatorial, weight-preserving bijection between basis elements of $H^{*}(\Hess(N,h))$ and the set of $P_h$-tableaux of shape $(1^n)$. 
\end{theorem}

As an application, we use these bijections to develop an understanding of the Poincar\'{e} polynomial via an argument counting $P_h$-tableaux with an inversion statistic. 

In addition to the higher Specht basis for $H^{*}(\Hess(S,h))$ when $h=(h(1),n,\ldots,n)$, we show that there is a natural way of defining a basis which decomposes into a sum of permutation representations. This result directly connects to the Stanley-Stembridge conjecture, since the Frobenius character of a permutation representation (after applying the involution $\omega$ on symmetric function) yields an $e$-positive symmetric function.

Lastly, we show that many of the arguments for the Hessenberg function $h=(h(1),n,\ldots, n)$ also work for the so-called transpose Hessenberg function $h'=((n-1)^{n-m},n^m)$. The polynomial generators are defined in a similar matter, and share many of the properties from the initial case. 

\begin{theorem}\label{thm:E}
    The following sets form a higher Specht basis of $H^{*}(\Hess(S,h))$ when the Hessenberg function is $h=((n-1)^{n-m},n^m)$:
\begin{equation*}
    \begin{split}
    x_1^{i_1}x_2^{i_2}\cdots x_n^{i_n} \,\,\,\, &  \mbox{ which does not contain the factor } \prod_{\ell=n-m+1}^{n} x_{\ell}  \\
    x_{n-1}^{\ell_1}x_{n-2}^{\ell_2}\cdots x_1^{\ell_{n-1}} (y_{k+1}-y_1) \,\,\,\, & \mbox{ which does not contain the factor } \prod_{\ell=1}^{n-m} x_{\ell} 
    \end{split}
\end{equation*} running over all $0\leq i_j\leq j-1$ in the first equation, and over all $0\leq \ell_j \leq j-1$ and $2\leq k\leq n$ in the second equation. 
\end{theorem}

\subsection{Outline}

In Section \ref{sec:background}, we give the necessary background on chromatic symmetric functions, Hessenberg varieties, and the representation theory of $\Sn$ in order to understand our results. We also define the Specht polynomials, coming from irreducible $\Sn$-modules, which motivate our first result.

In Section \ref{sec:transposebasis}, we look at the process Abe, Horiguchi, and Masuda used for finding a basis for $H^*(\Hess(S,h))$ when $h=(h(1),n,\ldots,n)$. We define a new class of GKM graphs, and show that this can be extended to a basis when $h' = ((n-1)^{n-m},n^m)$. We call this the transpose Hessenberg function, since the corresponding Dyck path is a reflection of the original Dyck path across a diagonal.

In Section \ref{sec:higherspecht}, we define an alternative basis for $H^{*}(\Hess(S,h))$ when $h=(h(1),n,\ldots,n)$, and show that this is a higher Specht basis, which directly illustrates the decomposition of $H^*(\Hess(S,h))$ into irreducible $\Sn$ representations. 

In Section \ref{sec:permutationbasis}, we show that the monomials from Theorem \ref{thm:A} can be reorganized into a permutation basis, which gives a geometric proof of the known result that the corresponding incomparability graphs have $e$-positive chromatic symmetric function. 

In Section \ref{sec:bijections}, we display a bijection between the fixed monomials in the higher Specht basis, which correspond to trivial representations of $\Sn$, and the set of $P_h$-tableaux of shape $(1^n)$, which correspond to the Schur function $s_{(n)}$ in Gasharov's expansion of $\omega X_{G}(\xvars)$. We also display a bijection between the second set of basis elements, which correspond to standard representations of $\Sn$, and the set of $P_h$-tableaux of shape $(2,1^{n-2})$, which correspond to the Schur function $s_{(n-1,1)}$.

Finally, in Section \ref{sec:poincarepolynomial}, we use the Shareshian-Wachs generalization of Gasharov's $P$-tableaux formula to give an alternative method for finding the Poincar\'{e} polynomial of $H^{*}(\Hess(S,h))$ for these Hessenberg functions.

\subsection{Acknowledgements}

I express my gratitude to Maria Gillespie, John Shareshian, and the Colorado State University Geometry Seminar for helpful discussions related to this work.

This work appeared as an extended abstract for the 36th International Convention on Formal Power Series and Algebraic Combinatorics \cite{Salois}.

\section{Background} \label{sec:background}

\subsection{Hessenberg Varieties}

In this paper, we use the notation for Hessenberg varieties which can be found in \cite{AHM}. Define $[n]$ to be the set $\{1,2,\ldots, n\}$. A function $h:[n]\to [n]$ is called a \textbf{Hessenberg function} of length $n$ if \begin{itemize}
    \item For all $i\in [n]$, $i\leq h(i)$, and
    \item For all $i\in [n-1]$, $h(i)\leq h(i+1)$.
\end{itemize}

We often write a Hessenberg function as a tuple $h = (h(1),h(2),\ldots, h(n))$. One can associate to each Hessenberg function a lattice path in the first quadrant from $(0,0)$ to $(n,n)$, called a Dyck path, by associating $h(i)$ with the height of the $i$-th horizontal step. Given a Hessenberg function $h$ and an $n\times n$ matrix $X$, we define the \textbf{Hessenberg variety} as follows:
\begin{equation*}
    \Hess(X,h) := \{V_{\bullet}\in \mathrm{Fl}(\C^n)\,|\,X(V_i) \subseteq V_{h(i)} \,\,\mathrm{for all}\,\,1\leq i\leq n\}
\end{equation*} where $\mathrm{Fl}(\C^n)$ is the flag variety consisting of sequences $V_{\bullet} = (V_0\subset V_1\subset \cdots \subset V_{n-1}\subset V_{n})$ of linear subspaces of $\C^n$ with $\dim(V_i) = i$ for all $i$. 

An alternate, yet equivalent definition is as follows. Let $\mathrm{GL}_n(\C)$ be the general linear group, and let $B$ be the Borel subgroup of invertible upper-triangular matrices. For a Hessenberg function $h$ of length $n$, let $H(h)$ be the set of matrices $M$ such that $M_{i,j} = 0$ whenever $i>h(j)$. For an $n\times n$ matrix $X$, the Hessenberg variety can be expressed as the set of cosets: $$\Hess(X,h) = \{gB\in \mathrm{Gl}_n(\C) /B \,\,|\,\, g^{-1}Xg \in H(h)\}.$$

Notice that, in the coset model for the Hessenberg variety, there is a natural invertible map from $\Hess(X,h)\to \Hess(m^{-1} X m, h)$ for any $m\in \mathrm{Gl}_n(\C)$, given by $gB \mapsto mgB$. Hence $\Hess(X,h) \cong \Hess(m^{-1}Xm, h)$. 

In \cite{Ty05}, Tymoczko proved that Hessenberg varieties behave nicely as varieties:

\begin{proposition}[\cite{Ty05}]
    Every Hessenberg variety $\Hess(X,h)$ admits an affine paving, that is, can be expressed as the disjoint union of a countable number of affine algebraic subvarieties. In particular, the cohomology ring $H^{*}(\Hess(X,h))$ is torsion-free and vanishes in the odd degree. 
\end{proposition}

As a result, we can write the Poincar\'{e} polynomial of $\Hess(X,h)$ in the following way: $$ \mathrm{Poin}(\Hess(X,h)),q) = \sum_{i=0}^{d} \mathrm{dim}(H^{2i}(\Hess(X,h))\,q^i $$ where $d$ is the dimension of $\Hess(X,h)$, and $\mathrm{deg}(q)=2$. This expression will be important for connecting the combinatorics of chromatic symmetric functions to the geometry of Hessenberg varieties. 

Let $h$ be a Hessenberg function of length $n$, and let $h'$ be the Hessenberg function obtained by reflecting the corresponding Dyck path for $h$ across the anti-diagonal. In other words, $h'(i) = |h^{-1}(\{n,n-1,\ldots,n+1-i\})|$. We call this the \textbf{transpose} Hessenberg function. The following observation is thanks to John Shareshian \cite{Shareshian}:

\begin{proposition} \label{prop:Transpose}
The space $\Hess(X,h)$ is isomorphic to $\Hess(X^T,h')$. 
\end{proposition}

\begin{proof}
    Let $W$ be the permutation matrix for the permutation $\pi(i)=n+1-i$ for all $i\in [n]$, and so $W = W^{-1} = W^T$. Notice that for any $n\times n$ matrix $X$, we have that $WXW$ is the matrix whose entries are the same as $X$, but rotated $180^{\circ}$, which we call $X^{R}$. Let $g\in \mathrm{GL}_n(\C)$. We will show that if $gB\in \Hess(X,h)$, then $((gW)^{-1})^T B\in \Hess(X^T,h')$: \begin{align*}
        (gW)^T X^T ((gW)^{-1})^T &= W g^T X^T (g^{-1})^{T} W\\
        &= W (g^{-1} X g)^T W\\
        &= ((g^{-1} X g)^T)^R
    \end{align*}

    Since $gB \in \Hess(X,h)$, we know $g^{-1}Xg \in H(h)$, so $(g^{-1}Xg)_{i,j} = 0$ whenever $i>h(j)$. Taking the transpose, we get that $(g^{-1} X g)^T_{i,j} = 0$ whenever $j>h(i)$. Finally, rotating by $180^{\circ}$, we have $((g^{-1} X g)^T)^R_{i,j} = 0$ when $i>h'(j)$. Hence $((gW)^{-1})^T B\in \Hess(X^T,h')$. Further note that $(WBW)^T=B$.  

    Since the map $g\mapsto ((gW)^{-1})^T$ is an involution, and each entry in the resulting matrix is a polynomial in the entries of $g$, this is an invertible algebraic morphism $\Hess(X,h) \to \Hess(X^T,h')$, and so the varieties are isomorphic.
\end{proof}

\begin{corollary}
    When $X$ is a semisimple matrix, $\Hess(X,h)$ is isomorphic to $\Hess(X,h')$.
\end{corollary}

\begin{proof} Since $X$ is diagonalizable, there exists some $m\in \mathrm{GL}_n(\C)$ such that $m^{-1}X m$ is diagonal. Since Hessenberg varieties are preserved under conjugation of the underlying matrix by an invertible matrix, we get the following: 
    $$\Hess(X,h) \cong \Hess(m^{-1} Xm, h) \cong \Hess((m^{-1} Xm)^T, h') = \Hess(m^{-1}Xm, h') \cong \Hess(X,h')$$ as desired.
\end{proof}

In this paper, we focus on Hessenberg varieties $\Hess(S,h)$ where $S$ is a regular semisimple matrix, meaning that $S$ is diagonalizable and has distinct eigenvalues. In particular, these Hessenberg varieties are smooth, and for a fixed Hessenberg function $h$, $\Hess(S,h)\cong \Hess(S',h)$ for any regular semisimple $S$ and $S'$ \cite{DePrSh}.  

\subsection{Chromatic Symmetric Functions}

In \cite{StanCSF}, Stanley defined the \textbf{chromatic symmetric function} of a finite graph $G$. This definition was later generalized by Shareshian and Wachs in \cite{SW} to form a quasisymmetric function: If $G$ has an (ordered) vertex set $V=\{v_1,\ldots, v_d\}$, then the \textbf{chromatic quasisymmetric function} is 
$$X_G(\xvars ; q) = \sum_{\kappa:V\to\N} x_{\kappa(v_1)}x_{\kappa(v_2)}\cdots x_{\kappa(v_d)}\,q^{\asc(\kappa)} $$
where the sum is over all proper colorings $\kappa:V\to \N$ of the vertices of $G$, and where $\asc(\kappa)$ is the number of vertices $v_i<v_j$ such that $\kappa(v_i)<\kappa(v_j)$.  Each term in the sum has the form $x_1^{i_1}\cdots x_n^{i_n}\,q^{\asc(\kappa)}$, where $i_k$ is the number of times $k$ was used in the coloring $\kappa$, and $n$ is the largest color used. 

One of the main conjectures of study for chromatic symmetric functions comes from Stanley and Stembridge in \cite{StSt}, but also naturally extends to the quasisymmetric case. Let $h$ be a Hessenberg function, and define the poset $P_h$ on the set $[n]$ as follows: We say $i<_{P_h} j$ if and only if $h(i)<j$. These posets are $(2+2)$- and $(3+1)$- free, so their incompatibility graphs relate to the original Stanley-Stembridge conjecture \cite{GP13}. The following generalization of the conjecture is equivalent to the one given in \cite{SW} by Shareshian and Wachs.

\begin{conjecture}[\cite{SW}, Conjecture 1.3]\label{conj:StanleyStembridge}
    If $G$ is the incomparability graph of the poset $P_h$ for some Hessenberg function $h$, then $X_G(\xvars; q)$ is $e$-positive and $e$-unimodal. 
\end{conjecture}

Shareshian and Wachs showed that, for the incomparability graphs $G = \inc(P_h)$, the chromatic quasisymmetric function is in fact symmetric, although this is not always true (for instance, any directed claw graph has non-symmetric chromatic quasisymmetric function). 

In \cite{Gasharov}, Gasharov proved a weaker condition -- that the chromatic symmetric function for this class of graphs is Schur-positive. This result was also extended by Shareshian and Wachs in \cite{SW} as follows.

Let $G$ be the incomparability graph of a poset $P$ on the set $[n]$, and let $\lambda$ be a partition of $n$. A \textbf{$P$-tableau} $T$ of shape $\lambda$ is a filling of the Young diagram of $\lambda$ with entries of $P$ such that the following conditions hold:
\begin{itemize}
    \item Each entry is used exactly once;
    \item If $j\in P$ appears directly to the right of $i\in P$, then $j >_P i$;
    \item If $j\in P$ appears directly above $i\in P$, then $j \not<_P i$.
\end{itemize}

We say a \textbf{$P$-inversion} in a $P$-tableau is a pair of entries $i,j$ such that $i<j$ as integers, $i$ appears in a higher row than $j$, and $i$ and $j$ are incomparable in $P$. Define $\Inv_P(T)$ to be the set of $P$-inversions of $T$, and $\inv_P(T):= |\Inv_P(T)|$ to be the number of $P$-inversions.

\begin{proposition}[\cite{SW}, Theorem 6.3]\label{prop:CSFexpansionPtab}
Let $G$ be the incomparability graph of a $(3+1)$-free poset $P$. Then $$ X_G(\xvars; q) = \sum_{\lambda\vdash n} \left( \sum_{T\in \mathrm{PT}(\lambda)} q^{\inv_P(T)} \right) s_{\lambda} $$ where $\mathrm{PT}(\lambda)$ is the set of $P$-tableaux of shape $\lambda$ and $s_{\lambda}$ is the Schur function for the partition $\lambda$.
\end{proposition}

\begin{example}\label{ex:simplePtableaux}
    Consider the poset $P$ on $\{1,2,3,4,5\}$ given by the single relation $1<_P 5$. The graph $G=\inc(P)$ is the complete graph on $\{1,2,3,4,5\}$ without the edge from $1$ to $5$. There are six $P$-tableaux of shape $(2,1,1,1)$, given below, with $3,4,4,5,5,$ and $6$ inversions, respectively. These $P$-tableaux contribute $(q^3+2q^4+2q^5+q^6)s_{(2,1,1,1)}$ to the Schur expansion of $X_G(\xvars;q)$. 

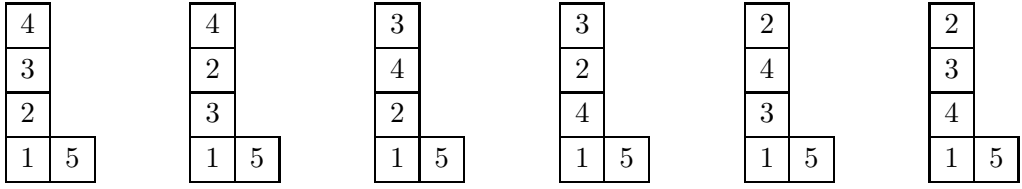
    \begin{center}
    \begin{ytableau}4 \\ 3 \\ 2 \\ 1 & 5 \end{ytableau} \hspace{1cm}
    \begin{ytableau}4 \\ 2 \\ 3 \\ 1 & 5 \end{ytableau} \hspace{1cm}
    \begin{ytableau}3 \\ 4 \\ 2 \\ 1 & 5 \end{ytableau} \hspace{1cm}
    \begin{ytableau}3 \\ 2 \\ 4 \\ 1 & 5 \end{ytableau} \hspace{1cm}
    \begin{ytableau}2 \\ 4 \\ 3 \\ 1 & 5 \end{ytableau} \hspace{1cm}
    \begin{ytableau}2 \\ 3 \\ 4 \\ 1 & 5 \end{ytableau}
    \end{center} \vspace{0.1cm}
\end{example}

\subsection{GKM Graphs}

Suppose $S$ is a regular semisimple matrix, so $S$ is diagonalizable with distinct eigenvalues. Then $\Hess(S,h)$ is smooth, and all Hessenberg varieties of this form are isomorphic to each other for a fixed $h$. In \cite{Ty07}, Tymoczko exhibited an action of $\Sn$ on $H^{*}(\Hess(S,h))$ using GKM theory, which realizes the equivariant cohomology of Hessenberg varieties as follows:

\begin{proposition}[\cite{Ty07}, Proposition 4.7]\label{prop:eqcohoring}
    The equivariant cohomology ring $H^{*}_T(\Hess(S,h))$ is isomorphic (as rings) to 
    \begin{align*}
        \left\{ \alpha\in \bigoplus_{w\in \Sn} \C[t_1,\ldots, t_n] \, {\bigg|} t_{w(i)}-t_{w(j)}\textrm{ divides } \alpha(w)-\alpha(w') \textrm{ if } w'= w(ji) \textrm{ for some } j<i \leq h(j) \right\} \label{eq:rsscohom}
    \end{align*}
    where $\alpha(w)$ is the $w$-component of $\alpha$ and $(ji)\in \Sn$ is the transposition of $j$ and $i$. We call the division criterion for the polynomials in two parts of a tuple the \textbf{GKM condition}.
\end{proposition}

Notice that, for each $i = 1,\ldots, n$, the tuple ${t}_i:= (t_i)_{w\in \Sn}$ is an element of the above set. From the theory of torus-equivariant cohomology (see \cite{AHM} for details), we get a presentation of the cohomology ring of $\Hess(S,h)$:
$$ H^{*}(\Hess(S,h);\C) \cong H^{*}_T(\Hess(S,h);\C)/({t}_1,\ldots,{t}_n) $$

The $\Sn$ action described by Tymoczko acts as follows: For $v\in \Sn$, and $\alpha = (\alpha(w))_{w\in \Sn} \in \bigoplus_{w\in \Sn} \C[t_1,\ldots, t_n]$ define $v\cdot \alpha$ by $(v\cdot \alpha)(w) = v\cdot \alpha(v^{-1}w)$ for all $w\in \Sn$, where the action on the right permutes the order of the tuple by the permutation $v$. 

Rather than thinking of elements of this ring as tuples of polynomials, it is often useful to think of them as labeled graphs with vertex set $\Sn$, and edges whenever $w' = w(ij)$ with $j<i\leq h(j)$, where the vertex label at $w$ is $\alpha(w)$.

\begin{figure}[ht] 
\begin{subfigure}{0.45\textwidth} \centering
\begin{tikzpicture}
    \tikzstyle{vertex} = [circle,draw=black, inner sep=2pt,fill=black];
    \node[vertex, label=below:{123}] (a) at (0,0) {};
    \node[vertex, label=right:{132}] (b) at (1,0.5) {};
    \node[vertex, label=right:{312}] (c) at (1,1.5) {};
    \node[vertex, label=left:{213}] (f) at (-1,0.5) {};
    \node[vertex, label=left:{231}] (e) at (-1,1.5) {};
    \node[vertex, label=above:{321}] (d) at (0,2) {};
    \draw[dashed] (b)--(c); \draw[dashed] (e)--(f);
    \draw[double distance = 2pt] (a)--(b); \draw[double distance = 2pt] (e)--(d);
    \draw (c)--(d); \draw (f)--(a);
\end{tikzpicture} \caption{$h=(2,3,3)$}
\end{subfigure}
\begin{subfigure}{0.45\textwidth} \centering
\begin{tikzpicture}
    \tikzstyle{vertex} = [circle,draw=black, inner sep=2pt,fill=black];
    \node[vertex, label=below:{123}] (a) at (0,0) {};
    \node[vertex, label=right:{132}] (b) at (1,0.5) {};
    \node[vertex, label=right:{312}] (c) at (1,1.5) {};
    \node[vertex, label=left:{213}] (f) at (-1,0.5) {};
    \node[vertex, label=left:{231}] (e) at (-1,1.5) {};
    \node[vertex, label=above:{321}] (d) at (0,2) {};
    \draw[dashed] (b)--(c); 
    \draw[dashed] (e)--(f); 
    \draw[dashed] (a)--(d);
    \draw[double distance = 2pt] (a)--(b); 
    \draw[double distance = 2pt] (e)--(d); 
    \draw[double distance = 2pt] (c)--(f);
    \draw (c)--(d); \draw (f)--(a); \draw (b)--(e);
\end{tikzpicture} \caption{$h=(3,3,3)$}
\end{subfigure}
\caption{The (unlabeled) GKM graphs for two Hessenberg functions. Dashed lines represent the permutation $(13)$, solid lines represent $(12)$, and doubled lines represent $(23)$.} \label{fig:GMKexample}
\end{figure}
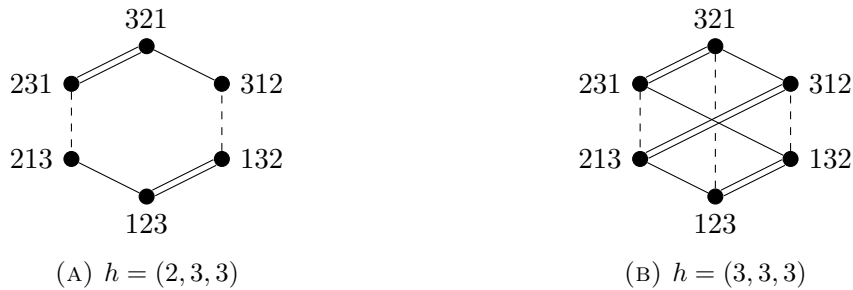

The following theorem ties together the geometry of Hessenberg varieties with the algebraic structure of chromatic symmetric functions. It was originally conjectured in \cite{SW} by Shareshian and Wachs, and proved both by Brosnan and Chow in \cite{BC18}, and by Guay-Paquet in \cite{GP16}.

\begin{proposition}[\cite{BC18}, \cite{GP16}] \label{thm:BC129}
Let $h$ be a Hessenberg vector, $P_h$ be the associated poset with incomparability graph $G=(V,E)$. Let $\Hess(S,h)$ be the regular semisimple Hessenberg variety for $h$. Then 
$$ \omega X_G(\mathbf{x};q) = \sum_{j=0}^{|E|} \mathrm{Frob}(H^{2j}(\Hess(S,h)))q^j, $$
where $\mathrm{Frob}(H^{2j}(\Hess(S,h)))$ is the Frobenius characteristic of the representation of $\Sn$ on the $2j^{th}$ cohomology group of $\Hess(S,h)$ and $\omega$ is the involution on symmetric functions that sends the Schur function $s_{\lambda}$ to $s_{\lambda'}$, where $\lambda'$ is the conjugate partition to $\lambda$. 
\end{proposition}

Combining Propositions \ref{prop:CSFexpansionPtab} and \ref{thm:BC129}, we can connect the graded representation of the cohomology ring for Hessenberg varieties with $P$-tableaux as follows. 
\begin{equation}\label{eq:HessPtabConnection} \sum_{j=0}^{|E|} \mathrm{Frob}(H^{2j}(\Hess(S,h)))q^j = \omega X_G(\mathbf{x};q) = \sum_{\lambda\vdash n}\left( \sum_{T\in PT(\lambda)} q^{\inv_h(T)}\right) s_{\lambda'} \end{equation}

From this, polynomial presentations of the cohomology ring $H^{*}(\Hess(S,h))$ can be connected explicitly to sets of $P_h$-tableaux, by understanding the action of $\Sn$ and its decomposition into irreducible representations. For example, in Section \ref{sec:bijections}, we give explicit bijections between monomials in the cohomology ring for $h=(h(1),n,\ldots, n)$ with $P_h$-tableaux of shapes $(1^n)$ and $(2,1^{n-2})$.

\subsection{Symmetric Group Representations}

Irreducible representations of $\Sn$ are indexed by partitions of $n$, and we write them as $V_{\lambda}$. The trivial representation is indexed by $\lambda = (n)$, and the standard representation is indexed by $\lambda = (n-1,1)$. For more details on the following facts, and a general overview of $\Sn$ representation theory, see \cite{SaganSymmetric}. 

Given a partition $\lambda$ of $n$, a \textbf{standard tableau} of shape $\lambda$ is a $P$-tableau of shape $\lambda$ where $P$ is the total ordering on $[n]$. In other words, each row and column is strictly increasing. Given a standard tableau $T$ of shape $\lambda$, define the \textbf{Specht polynomial} $F_T$ as $$ F_T = \prod_{C}\prod_{\{i<j\}\in C} (x_j-x_i)$$ where the outer product is over all columns of $T$. In Example \ref{ex:simplePtableaux}, the leftmost $P$-tableau is also a standard tableau, and the corresponding Specht polynomial is $$F_T = (x_4-x_3)(x_4-x_2)(x_4-x_1)(x_3-x_2)(x_3-x_1)(x_2-x_1) $$ 

Given a fixed partition $\lambda$ of $n$, the subspace of $\C[x_1,\ldots, x_n]$ generated by $\{F_T\}_{T\in \SYT(\lambda)}$, where $\SYT(\lambda)$ is the set of standard tableaux $T$ of shape $\lambda$, is isomorphic to the irreducible $\Sn$-module $V_{\lambda}$. These submodules are called the \textbf{Specht modules}. $\Sn$ acts on a Specht module by permuting the indices: For $w\in \Sn$, $w\cdot (x_j-x_i)=(x_{w(j)}-x_{w(i)})$. 

\begin{definition}[\cite{GR21}, Definition 1.5]
    If $R$ is an $\Sn$-module which decomposes into irreducible $\Sn$-modules as $$R = \bigoplus_{\lambda} c_{\lambda}V_{\lambda}\,,$$ then a \textbf{higher Specht basis of $R$} is a set of elements $\mathcal{B}$ with a decomposition $\mathcal{B} = \bigcup_{\lambda} \bigcup_{i=1}^{c_\lambda} \mathcal{B}_{i,\lambda}$ such that the elements of $\mathcal{B}_{i,\lambda}$ are a basis of the $i$-th copy of $V_{\lambda}$ in the decomposition of $R$.  
\end{definition}

Higher Specht bases for $\Sn$-modules have been constructed in the context of the coinvariant ring \cite{AST97} and the generalized coinvariant ring \cite{GR21}. These bases are valuable since they give a natural grouping of basis elements into irreducible modules. 

A representation $V$ of $\Sn$ is a \textbf{permutation representation} if there exists a set $A$ acted on by $\Sn$ such that $V\cong \C A$, the vector space of $\C$-linear combinations of elements of $A$. In other words, elements of $\Sn$ act via permutation matrices on $\C A$. 

If the trivial representation of $\Sn$ is $V_{(n)}$, then the \textbf{natural permutation representation} of $\Sn$ is the induced representation $V_{(n)} \uparrow_{\mathfrak{S}_{n-1}\times \mathfrak{S}_1}^{\Sn}$. Here, $\Sn$ acts on a set of $n$ disjoint cosets of the subgroup $\mathfrak{S}_{n-1}\times \mathfrak{S}_1$, so this representation is isomorphic to the $\Sn$-module $\C^n$, where $\Sn$ acts by permuting the standard unit vectors $\vec{e_1},\ldots,\vec{e_n}$. The natural permutation representation decomposes into irreducibles as $V_{(n)}\oplus V_{(n-1,1)}$. Then we have: $$\mathrm{Frob}(V) = \mathrm{Frob}(V_{(n)})+\mathrm{Frob}(V_{(n-1,1)}) = s_{(n)} + s_{(n-1,1)} = h_{(n-1,1)}. $$
Notice that $\omega(h_{(n-1,1)}) = e_{(n-1,1)}$, so if an $\Sn$-module $V$ is a direct sum of natural permutation representations (and trivial representations, which have Frobenius character $s_{(n)}=h_{(n)}$), then we get the $h$-positivity of the character of this representation. In general, permutation representations yield $h$-positive symmetric functions. When the $\Sn$-module is $H^*(\Hess(S,h))$, we can apply $\omega$ to get the $e$-positivity of the associated chromatic symmetric function.

\section{Basis Elements when $h'=((n-1)^{n-m},n^m)$}\label{sec:transposebasis}

In this section, we introduce the results of Abe, Horiguchi, and Masuda in \cite{AHM}, including their realization of $H^*(\Hess(S,h))$ as a polynomial quotient ring, in the case that $h=(h(1),n,\ldots,n)$. We then define our basis for $h'=((n-1)^{n-m},n^m)$. 

The authors in \cite{AHM} defined the following elements of $H^{*}(\Hess(S,h))$ in terms of GKM graphs: Let $h=(h(1),n,\ldots, n)$. For $1\leq k\leq n$, and $w\in \Sn$: \begin{align*}
    x_k(w) &:= t_{w(k)} \\
    y_k(w) &:= \begin{cases} \prod_{\ell=2}^{h(1)}(t_k-t_{w(\ell)}) & \mathrm{if}\,\, w(1) = k \\ 0 & \mathrm{if}\,\, w(1)\neq k         
    \end{cases}
\end{align*}

\begin{figure}[ht] 
\begin{subfigure}{0.45\textwidth} \centering
\begin{tikzpicture}
    \tikzstyle{vertex} = [circle,draw=black, inner sep=2pt,fill=black];
    \node[vertex, label=below:{$t_2$}] (a) at (0,0) {};
    \node[vertex, label=right:{$t_3$}] (b) at (1,0.5) {};
    \node[vertex, label=right:{$t_1$}] (c) at (1,1.5) {};
    \node[vertex, label=left:{$t_1$}] (f) at (-1,0.5) {};
    \node[vertex, label=left:{$t_3$}] (e) at (-1,1.5) {};
    \node[vertex, label=above:{$t_2$}] (d) at (0,2) {};
    \draw[dashed] (b)--(c); 
    \draw[dashed] (e)--(f);
    \draw[double distance = 2pt] (a)--(b); 
    \draw[double distance = 2pt] (e)--(d);
    \draw (c)--(d); 
    \draw (f)--(a);
\end{tikzpicture} \caption{The class $x_2$.}
\end{subfigure}
\begin{subfigure}{0.45\textwidth} \centering
\begin{tikzpicture}
    \tikzstyle{vertex} = [circle,draw=black, inner sep=2pt,fill=black];
    \node[vertex, label=below:{$0$}] (a) at (0,0) {};
    \node[vertex, label=right:{$0$}] (b) at (1,0.5) {};
    \node[vertex, label=right:{$0$}] (c) at (1,1.5) {};
    \node[vertex, label=left:{$t_2-t_1$}] (f) at (-1,0.5) {};
    \node[vertex, label=left:{$t_2-t_3$}] (e) at (-1,1.5) {};
    \node[vertex, label=above:{$0$}] (d) at (0,2) {};
    \draw[dashed] (b)--(c); 
    \draw[dashed] (e)--(f); 
    \draw[double distance = 2pt] (a)--(b); 
    \draw[double distance = 2pt] (e)--(d); 
    \draw (c)--(d); 
    \draw (f)--(a);
\end{tikzpicture} \caption{The class $y_2$.}
\end{subfigure}
\caption{Two GKM classes for $h=(2,3,3)$, as defined in \cite{AHM}.} \label{fig:AHMclasses}
\end{figure}
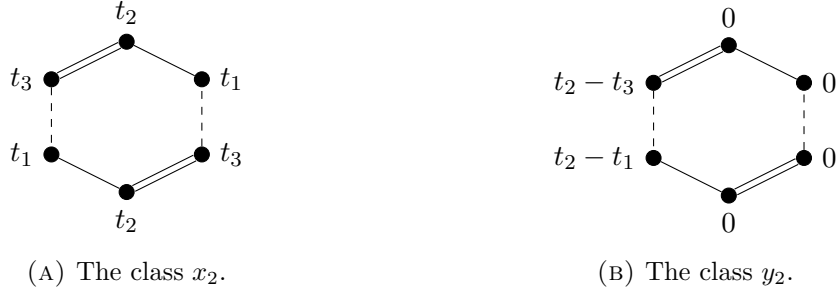

The authors showed that for each $k$, the tuples $x_k:=(x_k(w))_{w\in \Sn}$ and $y_k:=(y_k(w))_{w\in \Sn}$ satisfy the GKM condition. Further, they show that the $x_k$ and $y_{k}$ classes satisfy the following:

\begin{lemma}\label{lem:AHM4.1}(\cite{AHM}, Lemma 4.1)
    Suppose $h=(h(1),n,\ldots, n)$. The following hold: \begin{enumerate}
        \item $y_k\,y_{k'} = 0$ for $k\neq k'$.
        \item $x_1\,y_k = t_k\,y_k$ for all $k$.
        \item $y_k\,\prod_{\ell=h(1)+1}^{n}(t_k-x_{\ell}) = \prod_{\ell=2}^{n}(t_k-x_{\ell})$ for all $k$.
        \item $\sum_{k=1}^n\,y_k =\prod_{\ell=2}^{h(1)}(x_1-x_{\ell})$.
    \end{enumerate}
        We use the convention $\prod_{\ell=n+1}^{n}(t_k-x_{\ell}) =1$ in equation $(3)$ and $\prod_{\ell=2}^{1}(x_1-x_{\ell})=1$ in equation $(4)$.
\end{lemma}

Using these relations, the authors show that the classes $t_k$, $x_k$, and $y_k$ generate $H^{*}_T(\Hess(S,h))$. In particular, via the homomorphism $H^{*}_T(\Hess(S,h))\to H^{*}(\Hess(S,h))$ which takes the quotient by the torus fixed points $\langle t_1,\ldots, t_n\rangle$, they show that $H^{*}(\Hess(S,h))$ has the following presentation:

\begin{proposition}[\cite{AHM}, Theorem 4.3]\label{thm:AHM4.3}
If $h = (h(1),n,\ldots,n)$, then the cohomology ring of the regular semisimple Hessenberg variety $\Hess(S,h)$ is given by $$ H^{*}(\Hess(S,h)) \cong \Z[x_1,\ldots,x_n,y_1,\ldots,y_n] / I $$
where $\deg(x_i) = 2$, $\deg(y_k) = 2(h(1)-1)$, and $I$ is the homogeneous ideal generated by the following types of elements:
\begin{enumerate}
    \item $y_k y_{k'}$ $(1\leq k\neq k' \leq n)$
    \item $x_1 y_k$ $(1\leq k\leq n)$
    \item $(\prod_{\ell=h(1)+1}^n (-x_{\ell}))y_k - \prod_{\ell=2}^n(-x_{\ell})$ $(1\leq k\leq n)$
    \item $\sum_{k=1}^n y_k - \prod_{\ell=2}^{h(1)} (x_1-x_{\ell})$
    \item The $i-$th elementary symmetric polynomial $e_i(x_1,\ldots,x_n)$ $(1\leq i\leq n)$
\end{enumerate}
\end{proposition}

Additionally, the authors use the relations in $I$ to identify a basis of this space:

\begin{proposition}[\cite{AHM}, Remark 4.5] \label{thm:AHM4.5}
From the proof of Proposition \ref{thm:AHM4.3}, the following two types of monomials form a $\Z$-basis of $H^{*}(\Hess(S,h))$ when $h = (h(1),n,\ldots,n)$:
\begin{equation}
    x_1^{i_1}x_2^{i_2}\cdots x_n^{i_n} \,\,\,\,  \mbox{ which does not contain the factor } \prod_{\ell=1}^{h(1)} x_{\ell}
\end{equation} \vspace{-0.5cm}
\begin{equation}
    x_n^{\ell_1}x_{n-1}^{\ell_2}\cdots x_2^{\ell_{n-1}} y_k \,\,\,\, \mbox{ which does not contain the factor } \prod_{\ell=h(1)+1}^{n} x_{\ell} 
\end{equation}
running over all $0\leq i_j \leq n-j$ in the first equation, and over all $0\leq \ell_j\leq n-1-j$, and $1\leq k\leq n-1$ in the second equation. 
\end{proposition}


\subsection{The Cohomology Basis for the Transpose Hessenberg Function} \,

Recall from Proposition \ref{prop:Transpose} that if $S$ is a regular semisimple matrix, then $\Hess(S,h)\cong \Hess(S,h')$, where $h'$ is the transpose Hessenberg function to $h$. When $h=(h(1),n,\ldots,n)$, the transpose Hessenberg function is $h' = ((n-1)^{n-h(1)},n^{h(1)})$. In particular, we are motivated to find a similar basis for $\Hess(S,h')$ as Abe, Horiguchi, and Masuda did for $\Hess(S,h)$. 

Building from the GKM graphs defined in \cite{AHM}, we define the following when the Hessenberg function is $h' = ((n-1)^{n-m},\,n^{m})$. 

\begin{definition}\label{def:TransposeClasses}
    For $1\leq k\leq n$, and $w\in \Sn$, define:
    \begin{align*}
    x_{k}(w) &:= t_{w(k)} \\ 
    y_{k}(w) &:= \begin{cases}
        \prod_{\ell = n-m+1}^{n-1} (t_{k}-t_{w(\ell)}) & \mathrm{if}\,\, w(n) = k \\
        0 & \mathrm{if}\,\, w(n) \neq k 
    \end{cases}
    \end{align*} 
\end{definition}

An example of these GKM graphs is given in Figure \ref{fig:KJSclasses}.

\begin{figure}[ht] 
\begin{subfigure}{0.45\textwidth} \centering
\begin{tikzpicture}
    \tikzstyle{vertex} = [circle,draw=black, inner sep=2pt,fill=black];
    \node[vertex, label=below:{$t_2$}] (a) at (0,0) {};
    \node[vertex, label=right:{$t_3$}] (b) at (1,0.5) {};
    \node[vertex, label=right:{$t_1$}] (c) at (1,1.5) {};
    \node[vertex, label=left:{$t_1$}] (f) at (-1,0.5) {};
    \node[vertex, label=left:{$t_3$}] (e) at (-1,1.5) {};
    \node[vertex, label=above:{$t_2$}] (d) at (0,2) {};
    \draw[dashed] (b)--(c); 
    \draw[dashed] (e)--(f);
    \draw[double distance = 2pt] (a)--(b); 
    \draw[double distance = 2pt] (e)--(d);
    \draw (c)--(d); 
    \draw (f)--(a);
\end{tikzpicture} \caption{The class $x_2$.}
\end{subfigure}
\begin{subfigure}{0.45\textwidth} \centering
\begin{tikzpicture}
    \tikzstyle{vertex} = [circle,draw=black, inner sep=2pt,fill=black];
    \node[vertex, label=below:{$0$}] (a) at (0,0) {};
    \node[vertex, label=right:{$t_2-t_3$}] (b) at (1,0.5) {};
    \node[vertex, label=right:{$t_2-t_1$}] (c) at (1,1.5) {};
    \node[vertex, label=left:{$0$}] (f) at (-1,0.5) {};
    \node[vertex, label=left:{$0$}] (e) at (-1,1.5) {};
    \node[vertex, label=above:{$0$}] (d) at (0,2) {};
    \draw[dashed] (b)--(c); 
    \draw[dashed] (e)--(f); 
    \draw[double distance = 2pt] (a)--(b); 
    \draw[double distance = 2pt] (e)--(d); 
    \draw (c)--(d); 
    \draw (f)--(a);
\end{tikzpicture} \caption{The class $y_2$.}
\end{subfigure}
\caption{Two GKM classes for $h=(2,3,3)$, as given in Definition \ref{def:TransposeClasses}.} \label{fig:KJSclasses}
\end{figure}
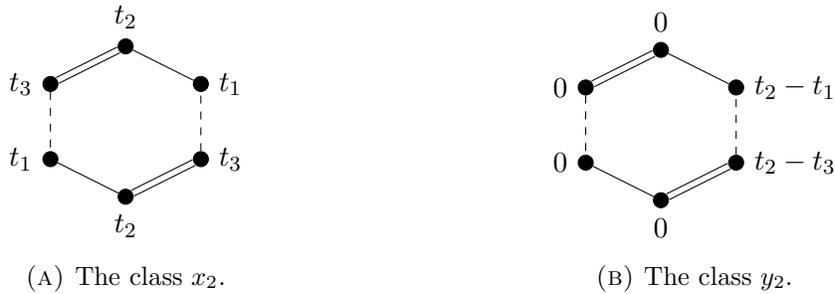

As above, define $x_k:= (x_k(w))_{w\in \Sn}$ and $y_k:= (y_k(w))_{w\in\Sn}$ First, we show that these tuples are elements of $H^{*}(\Hess(S,h'))$ by showing that they satisfy the GKM condition. 

\begin{lemma}
    Given $w\in \Sn$, if there exists $i<j\leq h'(i)$ such that $w' = w(ij)$, then $x_k(w) - x_k(w')$ and $y_{k}(w)-y_{k}(w')$ are divisible by $t_{w(i)}-t_{w(j)}$. 
\end{lemma}

\begin{proof}
    Suppose that we have $w\in \Sn$, and $i<j$ such that $j\leq h'(i)$ and $w' = w(ij)$. 
    
    First, we show that the GKM condition holds for all $x_{k}$. If $k=i$, then $x_{i}(w) - x_{i}(w') = t_{w(i)}-t_{w'(i)} = t_{w(i)} - t_{w(j)}$. Similarly, if $k=j$, then $x_j(w)-x_j(w') = t_{w(j)}-t_{w(i)}$. If $k\neq i$ and $k\neq j$, then $x_k(w) - x_k(w') = t_{w(k)}-t_{w'(k)} = t_{w(k)}-t_{w(k)} = 0$.  

    Now, we show that the GKM condition holds for all $y_k$. For the first case, suppose that $j<n$. If $w(n)\neq k$, then we have $w'(n)\neq k$, so $y_k(w)-y_k(w')=0-0 = 0$. If $w(n)=k$, then we have $w'(n)=k$. If $n-m+1\leq i$ or $j< n-m+1$, then $\prod_{\ell=n-m+1}^{n-1}(t_k-t_{w(\ell)}) = \prod_{\ell=n-m+1}^{n-1}(t_k-t_{w'(\ell)})$, since the index $\ell$ includes both or neither of $i$ and $j$. Hence $y_k(w) = y_k(w')$, so $y_k(w) - y_k(w') = 0$. If $i < n-m+1 \leq j$, then $y_k(w)$ contains the factor $(t_k-t_{w(j)})$, $y_k(w')$ contains the factor $(t_k-t_{w'(j)}) = (t_k-t_{w(i)})$, and all other factors are shared. Hence $(t_k-t_{w(j)})-(t_k-t_{w(i)}) = t_{w(i)}-t_{w(j)}$ divides $y_k(w) - y_k(w')$.

    For the second case, suppose that $j=n$, so $w(j) = w(n)$. Since $j\leq h'(i)$, we have that $h'(i) = n$, so $n-m+1\leq  i$. If $w(n)\neq k$ and $w'(n)\neq k$, then $y_k(w)-y_k(w')=0$. If $w(n)= k$, then $w' = w\cdot (in)$, so $w'(n)\neq k$. Thus $y_k(w)$ contains the factor $(t_k-t_{w(i)}) = (t_{w(j)}-t_{w(i)})$, and $y_{k}(w')=0$. Similarly, if $w'(n) = k$ and $w(n)\neq k$, then $y_{k}(w)=0$, and $y_k(w')$ contains the factor $(t_k-t_{w'(i)}) = (t_{w'(j)}-t_{w'(i)}) = (t_{w(i)}-t_{w(j)})$. Hence, in all cases, $t_{w(i)}-t_{w(j)}$ divides $y_k(w)-y_k(w')$.
\end{proof}

Since the $x_k$ and $y_k$ are well-defined elements of $H^*(\Hess(S,h'))$, we next find a set of similar relations to Lemma \ref{lem:AHM4.1}.

\begin{lemma} \label{lem:Equiv_properties}
    The following hold: \begin{enumerate}
        \item $y_k\, y_{k'} = 0$ for all $k\neq k'$.
        \item $x_n\,y_{k} = t_k\,y_k$ for all $k$.
        \item $y_k\,\prod_{\ell=1}^{n-m}(t_k-x_{\ell}) = \prod_{\ell=1}^{n-1}(t_k-x_{\ell})$ for all $k$.
        \item $\sum_{k=1}^{n}y_{k} = \prod_{\ell = n-m+1}^{n-1}(x_n-x_{\ell})$.
    \end{enumerate}
    We have the convention $\prod_{\ell=1}^{0}(t_k-x_{\ell}) = 1$ in equation $(3)$, and $\prod_{\ell=n}^{n-1}(x_n-x_{\ell}) = 1$ in equation $(4)$. 
\end{lemma}

\begin{proof} Let $w\in \Sn$.

    (1) If $k\neq k'$, then $w(n)=k$ or $w(n)=k'$ (or neither), but not both. So we have $y_{k}(w)y_{k'}(w)=0$.

    (2) If $w(n)=k$, then $x_n(w) = t_{w(n)}=t_k$. If $w(n)\neq k$, then $y_k(w)=0$. In either case, $x_n(w)y_k(w)=t_k\,y_k(w)$. 

    (3) If $w(n) = k$, then we have $$ y_k(w) \cdot \prod_{\ell=1}^{n-m}(t_k-x_{\ell}(w)) = \prod_{\ell=n-m+1}^{n-1}(t_k-t_{w(\ell)}) \cdot \prod_{\ell=1}^{n-m}(t_k-t_{w(\ell)}) = \prod_{\ell=1}^{n-1} (t_k-x_{\ell}(w))$$ If $w(n)\neq k$, then $y_k(w) = 0$. Further, we must have $w(i)=k$ for some $i<n$, and so $\prod_{\ell=1}^{n-1}(t_k-t_{w(\ell)})$ contains the factor $(t_k-t_k)$, and is thus $0$. 

    (4) Suppose $w(n)=k'$. Then for any $k\neq k'$, we have $y_k(w) = 0$. We have $$ \sum_{k=1}^{n} y_{k}(w) = \prod_{\ell= n-m+1}^{n-1}(t_{k'}-t_{w(\ell)}) = \prod_{\ell= n-m+1}^{n-1}(t_{w(n)}-t_{w(\ell)}) = \prod_{\ell= n-m+1}^{n-1}(x_{n}(w) - x_{\ell}(w)) $$
\end{proof}

Notice that the relations in Lemma \ref{lem:Equiv_properties} are the same as those in Lemma \ref{lem:AHM4.1}, with each $x_i$ swapped with $x_{n+1-i}$, and supposing that $m=h(1)$.  Under this isomorphism of polynomial rings, we obtain the following results.

\begin{proposition}
    If $h' = ((n-1)^{n-m},n^m)$, then the classes $x_k$, $y_k$, $t_k$ (for $1\leq k\leq n$) generate $H^{*}_T(\Hess(S,h'))$ as a $\Z$-algebra. 
\end{proposition}

We now give the quotient ring presentation of $H^{*}(\Hess(S,h'))$. 

\begin{theorem}
    If $h' = ((n-1)^{n-m},n^m)$, then the cohomology ring of $H^*(\Hess(S,h'))$ is given by \begin{equation*}
        H^*(\Hess(S,h')) \cong \Z[x_1,\ldots, x_n,y_1,\ldots,y_n] / I
    \end{equation*} where $\deg(x_k) = 2$, $\deg(y_k) = 2(m-1)$, and $I$ is the homogeneous ideal generated by the following types of elements: \begin{enumerate}
        \item $y_k\,y_{k'}$ for all $k\neq k'$.
        \item $x_n\,y_k$ for all $k$.
        \item $y_k\,\prod_{\ell=1}^{n-m}(-x_{\ell}) - \prod_{\ell=1}^{n-1}(-x_{\ell})$ for all $k$.
        \item $\sum_{k=1}^{n} y_k - \prod_{\ell=n-m+1}^{n-1} (x_n-x_{\ell})$. 
        \item The $i$-th elementary symmetric polynomial $e_i(x_1,\ldots, x_n)$ for $1\leq i\leq n$. 
    \end{enumerate}
\end{theorem}

\begin{remark}
    If $h'=((n-1)^{n-m},n^m)$, then following two types of monomials form a $\Z$-basis of $H^{*}(\Hess(S,h'))$:
    \begin{equation}\label{eq:KJS1}
    x_1^{i_1}x_2^{i_2}\cdots x_{n}^{i_n} \,\,\,\,  \mbox{ which does not contain the factor } \prod_{\ell=n-m+1}^{n} x_{\ell}
    \end{equation} \vspace{-0.5cm}
    \begin{equation}\label{eq:KJS2}
    x_{n-1}^{\ell_1}x_{n-2}^{\ell_2}\cdots x_{1}^{\ell_{n-1}}\,y_k \,\,\,\,  \mbox{ which does not contain the factor } \prod_{\ell=1}^{n-m} x_{\ell}
\end{equation}
running over all $0\leq i_j\leq j-1$ in the first equation, and over all $0\leq \ell_{j}\leq j-1$ and $1\leq k\leq n-1$ in the second equation. 
\end{remark}

\section{Higher Specht Basis for $H^*(\Hess(S,h))$} \label{sec:higherspecht}

In this section, we give a new basis for $H^*(\Hess(S,h))$ when  $S$ is regular semisimple and $h=(h(1),n,\ldots, n)$,  and show that it is a higher Specht basis. Recall the set of monomial basis elements of $H^{*}(\Hess(S,h))$ defined for $h=(h(1),n,\ldots,n)$ in the previous section:
\begin{equation}\label{eq:AHM1}
    x_1^{i_1}x_2^{i_2}\cdots x_n^{i_n} \,\,\,\,  \mbox{ which does not contain the factor } \prod_{\ell=1}^{h(1)} x_{\ell}
\end{equation} \vspace{-0.5cm}
\begin{equation}\label{eq:AHM2}
    x_n^{\ell_1}x_{n-1}^{\ell_2}\cdots x_2^{\ell_{n-1}} y_k \,\,\,\, \mbox{ which does not contain the factor } \prod_{\ell=h(1)+1}^{n} x_{\ell} 
\end{equation}
running over all $0\leq i_j \leq n-j$ in the first equation, and over all $0\leq \ell_j\leq n-1-j$, and $1\leq k\leq n-1$ in the second equation.

The action of $\Sn$ on the above polynomials is defined by fixing each of the $x_i$, and permuting the $y_i$, that is, for $w\in \Sn$, we have $w\cdot x_i = x_i$ and $w\cdot y_i = y_{w(i)}$. This group action gives a representation of $\Sn$, and it has been shown that this representation decomposes into the direct sum of trivial representations and standard representations. 

Define $B_1$ to be the set of monomials in equation (\ref{eq:AHM1}) above, and $B_2$ to be the set of monomials in equation (\ref{eq:AHM2}) above. We construct an alternate basis of $H^{*}(\Hess(S,h))$ motivated by the basis elements of Specht modules. Let $B_3$ be the set of monomials of the form 
\begin{equation}\label{eq:AHM3}
    x_n^{\ell_1}x_{n-1}^{\ell_2}\cdots x_2^{\ell_{n-1}}(y_{k+1}-y_1)\,\,\,\, \mbox{ which does not contain the factor } \prod_{\ell=h(1)+1}^{n} x_{\ell}
\end{equation}
running over all $0\leq \ell_j \leq n-1-j$ and $1\leq k\leq n-1$. 

From Proposition \ref{thm:AHM4.5}, the set $B_1\cup B_2$ forms a basis of $H^{*}(\Hess(S,h))$ when $h=(h(1),n,\ldots, n)$. We claim that $B_1\cup B_3$ also forms a basis. 

\begin{lemma}\label{lem:newbasislem1}
Let $f(x_1,\ldots, x_n)$ be homogeneous in this presentation of $H^{*}(\Hess(S,h))$. Then $f(x_1,\ldots, x_n)$ can be expressed only in terms of the elements of $B_1$.  
\end{lemma}

\begin{proof}
Since $B_1\cup B_2$ forms a basis of $H^{*}(\Hess(S,h))$, we can write 
$$ f(x_1,\ldots, x_n) = \sum_{B_1} b_{\underline{i}} x_1^{i_1}\cdots x_n^{i_n} + \sum_{B_2} c_{\underline{\ell},k} x_n^{\ell_1}\cdots x_2^{\ell_{n-1}} y_k $$ 
for some constants $b_{\underline{i}}$ and $c_{\underline{\ell},k}$ ranging over the lists of exponents $\underline{i}$ and $\underline{\ell}$ of elements of $B_1\cup B_2$, and with $1\leq k\leq n-1$. From the dot action of $\Sn$ on $H^{*}(\Hess(S,h))$, we know that for $w\in \Sn$, $w\cdot  x_i = x_i$ and $w\cdot y_i = y_{w(i)}$ for all $i$. So, since $f(x_1,\ldots, x_n)$ is fixed by the action of $\Sn$, we get that $c_{\underline{\ell},k} = c_{\underline{\ell},k'}$ for all $k\neq k'$. Define $c_{\underline{\ell}} := c_{\underline{\ell},k}$. 

If $c_{\underline{\ell}}\neq 0$, then for $w\in \Sn$, $w$ fixes $f(x_1,\ldots x_n)$, and fixes $\sum_{B_1} b_{\underline{i}} x_1^{i_1}\cdots x_{n}^{i_n}$, $w$ must also fix $\sum_{B_2} c_{\underline{\ell}} x_n^{\ell_1}\cdots x_2^{\ell_{n-1}}$.  Hence we can use the transposition $(i,j) \in \Sn$ to find: 
\begin{equation*}
    \begin{split}
        \sum_{k\neq i} c_{\underline{\ell}} x_n^{\ell_1}\cdots x_2^{\ell_{n-1}} y_k &= \sum_{k\neq j} c_{\underline{\ell}} x_n^{\ell_1}\cdots x_2^{\ell_{n-1}} y_k \\ 
        \sum_{k\neq i} c_{\underline{\ell}} x_n^{\ell_1}\cdots x_2^{\ell_{n-1}} y_k - \sum_{k\neq i,j} c_{\underline{\ell}} x_n^{\ell_1}\cdots x_2^{\ell_{n-1}} y_k &= \sum_{k\neq j} c_{\underline{\ell}} x_n^{\ell_1}\cdots x_2^{\ell_{n-1}} y_k - \sum_{k\neq i,j} c_{\underline{\ell}} x_n^{\ell_1}\cdots x_2^{\ell_{n-1}} y_k \\ 
        x_{n}^{\ell_1}\cdots x_{2}^{\ell_{n-1}} y_j &= x_{n}^{\ell_1}\cdots x_{2}^{\ell_{n-1}} y_i
    \end{split}
\end{equation*}
However, this is a contradiction, since these are distinct basis elements in $B_1\cup B_2$, so they are linearly independent. Hence $c_{\ell} = 0$, so $f(x_1,\ldots, x_n)$ can be expressed solely using basis elements from $B_1$.
\end{proof}

We now show that $B_1$ and $B_3$ together form a basis for $H^{*}(\Hess(S,h))$.

\begin{theoremA}
The following sets form a basis of $H^{*}(\Hess(S,h))$ when $h=(h(1),n,\ldots,n)$: 
\begin{equation*}\label{eq:newAHMbasis}
    \begin{split}
    x_1^{i_1}x_2^{i_2}\cdots x_n^{i_n} \,\,\,\, &  \mbox{ which does not contain the factor } \prod_{\ell=1}^{h(1)} x_{\ell}  \\
    x_n^{\ell_1}x_{n-1}^{\ell_2}\cdots x_2^{\ell_{n-1}} (y_{k+1}-y_1) \,\,\,\, & \mbox{ which does not contain the factor } \prod_{\ell=h(1)+1}^{n} x_{\ell} 
    \end{split}
\end{equation*}
running over all $0\leq i_j \leq n-j$ in the first equation, and over all $0\leq \ell_j\leq n-1-j$, and $1\leq k\leq n-1$ in the second equation. 
\end{theoremA}

\begin{proof}
Note that $B_1\cup B_2$ form the basis of $H^{*}(\Hess(S,h))$ given in Proposition \ref{thm:AHM4.5}, and our desired basis is $B_1\cup B_3$. Consider the degree-lexicographic ordering on the monomials given by $x_1<x_2<\ldots<x_n<y_2<\ldots<y_n<y_1$, and order the elements of $B_1,B_2$, and $B_3$ correspondingly, with elements of $B_3$ ordered by their initial monomial. We will form the transition matrix $M$ from $B_1\cup B_2$ to $B_1\cup B_3$, and show that it is invertible. Note that since $H^{*}(\Hess(S,h))$ is graded with $\deg(x_i) = 2$ and $\deg(y_i) = 2(h(1)-1)$ for all $i$, the transition matrix is block diagonal, with blocks $M_d$ according to degree $d$. The columns of $M_d$ correspond to ways of writing one element of $B_1\cup B_3$ of degree $d$ in terms of the basis elements from $B_1\cup B_2$. 

For $M_0$, there are two cases. If $h(1)>1$, then $M_0 =[ \begin{array}{c} 1  \end{array}]$, which is invertible. If $h(1)=1$, then $B_2$ contains elements $y_1,\ldots, y_{n-1}$ of degree $0$, and $B_3$ contains elements $y_2-y_1,\ldots y_n-y_1$ of degree $0$. For the ordering given above, the transition matrix is given in Figure \ref{fig:smallmatrix}. This matrix is invertible, since we can use row operations to subtract add rows $2$ through $n-1$ to the last row to get an upper-triangular matrix with nonzero entries on the diagonal.

\begin{figure}
    $$ M_0 = \left[\begin{array}{cccccc}
    1 & 0 & 0 & \cdots & 0 & 1  \\
    0 & 1 & 0 & \cdots & 0 & -1 \\
    0 & 0 & 1 & \cdots & 0 & -1 \\
    \vdots & \vdots & \vdots & \ddots & \vdots & \vdots \\
    0 & 0 & 0 & \cdots & 1 & -1 \\
    0 & -1 & -1 & \cdots & -1 & -2 \\
    \end{array} \right] $$ 
    \caption{The transition matrix between $B_1\cup B_2$ and $B_1\cup B_3$ for elements of degree $0$ when $h(1)=1$.}
    \label{fig:smallmatrix}
\end{figure}

When $d>0$, consider the transition matrix $M_d$. Columns corresponding to an element of $B_1$ will have a $1$ on the diagonal, and $0$'s elsewhere, since these basis elements exist in both sets. For $k<n$, we can write $x_n^{\ell_1}\cdots x_{2}^{\ell_{n-1}} (y_k-y_1) = x_n^{\ell_1}\cdots x_2^{\ell_{n-1}}y_k - x_n^{\ell_1}\cdots x_2^{\ell_{n-1}}y_1$, so columns corresponding to elements in $B_3$ of this form will have a $1$ on the diagonal, corresponding to the row for $x_n^{\ell_1}\cdots x_2^{\ell_{n-1}} y_k$, a $(-1)$ below the diagonal, corresponding to the row for $x_n^{\ell_1}\cdots x_2^{\ell_{n-1}} y_1$, and $0$'s elsewhere.

For elements in $B_3$ of the form $x_n^{\ell_1}\cdots x_2^{\ell_{n-1}} (y_n-y_1)$, we use relation $(4)$ from Proposition \ref{thm:AHM4.3} to rewrite: 
\begin{equation*} x_n^{\ell_1}\cdots x_2^{\ell_{n-1}} (y_n-y_1) \end{equation*}
\begin{equation*}
    = x_n^{\ell_1}\cdots x_2^{\ell_{n-1}} \left(-2y_1-y_2-\cdots - y_{n-1}\right) + x_n^{\ell_1}\cdots x_2^{\ell_{n-1}} \left( \prod_{m=2}^{h(1)} (x_1-x_m)\right) 
\end{equation*}

The first term of this expansion is clearly a sum of basis elements from $B_2$. The second term is a polynomial containing only the variables $x_1,\ldots, x_n$, so by Lemma \ref{lem:newbasislem1}, we know that it can be written as a sum of basis elements from $B_1$. So we have
\begin{equation*}
    \begin{split}
        & x_n^{\ell_1}\cdots x_2^{\ell_{n-1}} (y_n-y_1) \\
        & = x_n^{\ell_1}\cdots x_2^{\ell_{n-1}} \left(-2y_1-y_2-\cdots - y_{n-1}\right) + x_n^{\ell_1}\cdots x_2^{\ell_{n-1}} \left( \prod_{m=2}^{h(1)} (x_1-x_m)\right) \\
        & = x_n^{\ell_1}\cdots x_2^{\ell_{n-1}} \left(-2y_1-y_2-\cdots - y_{n-1}\right) + \sum_{B_1} b_{\underline{i}} x_1^{i_1}\cdots x_{n}^{i_n} \\
    \end{split}
\end{equation*}

Thus the column of $M_d$ corresponding to $x_{n}^{\ell_1}\cdots x_{n}^{\ell_{n-1}}(y_n - y_1)$ has a $b_{\underline{i}}$ in rows corresponding to elements from $B_1$, a $(-1)$ in the rows corresponding to $x_n^{\ell_1}\cdots x_{2}^{\ell_{n-1}} y_{k}$ for $k<n$ and a $(-2)$ in the row for $x_{n}^{\ell_1}\cdots x_{2}^{\ell_{n-1}}y_1$. So the matrix has the form shown in Figure \ref{fig:transition_matrix}. Then, using row operations, adding the row for the basis element $x_{n}^{\ell_1}\cdots x_{2}^{\ell_{n-1}} y_k$ to the row for the basis element $x_{n}^{\ell_1}\cdots x_{2}^{\ell_{n-1}} y_1$ clears the entries below the diagonal, and maintains that the diagonal is nonzero, because of the formula found above for expressing $x_{n}^{\ell_1}\cdots x_{2}^{\ell_{n-1}} (y_n-y_1)$ in terms of basis elements from $B_1\cup B_2$, namely that the coefficient for $x_{n}^{\ell_1}\cdots x_{2}^{\ell_{n-1}} y_k$ is $-1$ when $2\leq k \leq n-1$. 

Therefore we can use elementary row operations to make $M_d$ an upper-triangular matrix with nonzero entries on the diagonal. Hence each $M_d$ is invertible, so the transition matrix from $B_1\cup B_2$ to $B_1\cup B_3$ is invertible, and so $B_1\cup B_3$ forms a basis of $H^{*}(\Hess(S,h))$.
\end{proof}

\begin{figure}
    $$\left[\begin{array}{cccccccccccc}
    1 & 0 & \cdots & 0 & 0 & 0 & \cdots & b_{\underline{i}^{(1)}} & \cdots & 0 & \cdots & b_{\underline{j}^{(1)}}  \\
    0 & 1 & \cdots & 0 & 0 & 0 & \cdots & b_{\underline{i}^{(2)}} & \cdots & 0 & \cdots & b_{\underline{j}^{(2)}} \\
    \vdots & \vdots  & \ddots & \vdots & \vdots & \vdots &  & \vdots &  & \vdots &  & \vdots \\
    0 & 0 & \cdots & 1 & 0 & 0 & \cdots & b_{\underline{i}^{(k)}} & \cdots & 0 & \cdots & b_{\underline{j}^{(k)}} \\
    0 & 0 & \cdots & 0 & 1 & 0 & \cdots & -1 & \cdots & 0 & \cdots & 0 \\
    0 & 0 & \cdots & 0 & 0 & 1 & \cdots & -1 & \cdots & 0 & \cdots & 0 \\
    \vdots & \vdots &  & \vdots & \vdots & \vdots & \ddots & \vdots &  & \vdots &  & \vdots \\
    0 & 0 & \cdots & 0 & -1 & -1 & \cdots & -2 & \cdots & 0 & \cdots & 0 \\
    \vdots & \vdots &  & \vdots & \vdots & \vdots &  & \vdots &  & \vdots &  & \vdots \\
    0 & 0 & \cdots & 0 & 0 & 0 & \cdots & 0 & \cdots & 1 & \cdots & -1 \\
    \vdots & \vdots &  & \vdots & \vdots & \vdots &  & \vdots &  & \vdots & \ddots & \vdots \\
    0 & 0 & \cdots & 0 & 0 & 0 & \cdots & 0 & \cdots & -1 & \cdots & -2 \\
\end{array}\right] $$
    \caption{The general form of the transition matrix between $B_1\cup B_2$ and $B_1\cup B_3$.}
    \label{fig:transition_matrix}
\end{figure}

By counting the monomials in each set, we can use these basis elements to understand the multiplicity of irreducible representations in the decomposition of $H^*(\Hess(S,h))$. 

\begin{proposition}\label{cohomdecomp}
When $h=(h(1),n,\ldots,n)$, the dot action of $\Sn$ on $H^{*}(\Hess(S,h))$ decomposes into $h(1)(n-1)!$ copies of the trivial representation $V_{(n)}$ and $(n-h(1))(n-2)!$ copies of the standard representation $V_{(n-1,1)}$.
\end{proposition}

\begin{proof}
The dot action acts on the polynomial ring $\Z[x_1,\ldots,x_n, y_1,\ldots,y_n] / I \cong H^{*}(\Hess(S,h))$ as follows: For $w\in \Sn$, we have $w\cdot x_i = x_i$ and $w\cdot y_i = y_{w(i)}$. Since the set $B_1\cup B_3$ forms a basis, we can examine the dot action on these basis elements. 
By counting the possible strings of exponents on each type of monomial, we find that there are $h(1)(n-1)!$ basis elements in $B_1$ and $(n-h(1))(n-1)!$ basis elements from $B_3$. Each of the basis elements from $B_1$ is fixed by the dot action, so this representation includes $h(1)(n-1)!$ copies of the trivial representation.

Let $V_{(n-1,1)}$ be the Specht module for the standard representation, so $V_{(n-1,1)}$ has a basis of the form $\{y_k-y_1\}_{2\leq k\leq n}$. Consider a map $\varphi: B_3 \to V_{(n-1,1)}$  that maps a given polynomial as follows: $\varphi(x_{n}^{\ell_1}\cdots x_{2}^{\ell_{n-1}} (y_k-y_1)) = y_k-y_1$. We note that $\varphi$ preserves the action of $\Sn$, since for any $w\in \Sn$,
\begin{equation*}
    \begin{split}
    \varphi( w\cdot (x_{n}^{\ell_1}\cdots x_{2}^{\ell_{n-1}} (y_k-y_1)) &= \varphi((x_{n}^{\ell_1}\cdots x_{2}^{\ell_{n-1}} (y_{w(k)}-y_{w(1)}))  \\
    &= y_{w(k)}-y_{w(1)}  \\
    &= w\cdot (y_k-y_1) \\
    &= w\cdot \varphi (x_n^{\ell_1}\cdots x_{2}^{\ell_{n-1}} (y_k-y_1)) \\
    \end{split}
\end{equation*}

For a fixed list of exponents $\underline{\ell}$, the restriction of $\varphi$ to the elements of $B_3$ with a common factor $x_n^{\ell_1}\cdots x_2^{\ell_{n-1}}$ is a bijection between these polynomials and the basis $\{y_k-y_1\}_{2\leq k\leq n}$ of $V_{(n-1,1)}$. From above, we know that there are $(n-h(1))(n-1)!$ basis elements in $B_3$, so there are $\frac{(n-h(1))(n-1)!}{n-1} = (n-h(1))(n-2)!$ elements for each fixed exponent vector $\ell$. For each such exponent vector, the monomials $\{x_n^{\ell_1}\cdots x_2^{\ell_{n-1}}\,(y_k-y_1)\}_{2\leq k\leq n}$ span a submodule isomorphic to the Specht module $V_{(n-1)}$ for the standard representation. The monomials from $B_1$ each span a trivial representation, and there are $h(1)(n-1)!$ choices for exponent vector on each of these monomials. Hence $H^*(\Hess(S,h))$ decomposes into irreducibles as desired.
\end{proof}

In other words, we have succeeded in finding a higher Specht basis for the cohomology ring of $\Hess(S,h)$, where there is a natural projection from $H^{*}(\Hess(S,h))$ to $V_{(n)}$ and $V_{(n-1,1)}$. This mapping is given by $\pi_1(x_1^{i_1}\cdots x_{n}^{i_n}) = 1$ and $\pi_2(x_n^{\ell_1}\cdots x_{2}^{\ell_{n-1}} (y_k-y_1)) = y_k-y_1$.

By a similar argument to the above, we get a higher Specht basis for the cohomology ring of $\Hess(S, ((n-1)^{n-m},n^m)$, using the isomorphism which sends $x_i$ to $x_{n-i+1}$:

\begin{theoremE}
The following sets form a basis of $H^{*}(\Hess(S,h))$ when $h=((n-1)^{n-m},n^m)$:
\begin{equation*}\label{eq:transposeSpecht}
    \begin{split}
    x_1^{i_1}x_2^{i_2}\cdots x_n^{i_n} \,\,\,\, &  \mbox{ which does not contain the factor } \prod_{\ell=n-m+1}^{n} x_{\ell}  \\
    x_{n-1}^{\ell_1}x_{n-2}^{\ell_2}\cdots x_1^{\ell_{n-1}} (y_{k+1}-y_1) \,\,\,\, & \mbox{ which does not contain the factor } \prod_{\ell=1}^{n-m} x_{\ell} 
    \end{split}
\end{equation*} running over all $0\leq i_j\leq j-1$ in the first equation, and over all $0\leq \ell_j \leq j-1$ and $1\leq k\leq n-1$ in the second equation. 
\end{theoremE}

\section{Permutation Representations}\label{sec:permutationbasis}

In the previous section, we saw that for $h=(h(1),n,\ldots, n)$, the sets $B_1$ and $B_3$ formed a higher Specht basis of $H^{*}(\Hess(S,h))$. Each element of $B_1$ spans a trivial representation isomorphic to $V_{(n)}$, and each set of $n-1$ elements of $B_3$ with the same degree on each $x_i$ span a standard representation isomorphic to $V_{(n-1,1)}$. We now turn towards permutation representations, since these connect to the question of $e$-positivity for the chromatic symmetric function. 

\begin{proposition}
    Suppose $x_n^{\ell_1}x_{n-1}^{\ell_2}\cdots x_{2}^{\ell_{n-1}}$ does not contain the factor $\prod_{\ell=h(1)+1}^{n} x_{\ell}$, and for all $j$, $0\leq \ell_j\leq n-1-j$. Then the set $\{ x_n^{\ell_1}x_{n-1}^{\ell_2}\cdots x_{2}^{\ell_{n-1}}\,y_k \}_{1\leq k\leq n}$ forms a linearly independent set of elements of $H^{*}(\Hess(S,h))$.
\end{proposition}

\begin{proof}
    If $h(1)=n$, then there are no such elements defined above, so we assume that $h(1)<n$. Notice that it is equivalent to show that the set $$\{ x_n^{\ell_1}x_{n-1}^{\ell_2}\cdots x_{2}^{\ell_{n-1}}\,y_k\}_{1\leq k\leq n-1} \cup \{x_n^{\ell_1}x_{n-1}^{\ell_2}\cdots x_{2}^{\ell_{n-1}}\,(y_1+\cdots + y_n)\}$$ is linearly independent, since the transition matrix between the given set and this set is invertible. Since elements of $B_2$ form distinct basis elements, we have that the set of $n-1$ elements on the left is linearly independent. Using relations $(1)$ and $(2)$ from Proposition \ref{thm:AHM4.3}, $H^{*}(\Hess(S,h))$ does not contain elements divisible by $y_ky_{\ell}$ for $k\neq \ell$, or elements divisible by $x_1y_k$ for any $k$. Under relation $(3)$, we have that $0 = \prod_{\ell=2}^{n}(-x_{\ell}) - y_k (\prod_{\ell=h(1)+1}^{n}(-x_{\ell}))$. Notice that $\prod_{\ell=h(1)+1}^{n}(-x_{\ell})$ divides this term, but does not divide $x_n^{\ell_1}x_{n-1}^{\ell_2}\cdots x_{2}^{\ell_{n-1}}$ by definition, so this relation does not factor into our calculation of linear independence.
    
    Using relation $(4)$, we have the following equivalence in $H^{*}(\Hess(S,h))$: \begin{equation}\label{eq:onlyxmono}
         x_n^{\ell_1}x_{n-1}^{\ell_2}\cdots x_{2}^{\ell_{n-1}}(y_1+\cdots + y_n) = x_n^{\ell_1}x_{n-1}^{\ell_2}\cdots x_{2}^{\ell_{n-1}}\left( \prod_{\ell=2}^{h(1)}(x_1-x_{\ell}) \right). \end{equation}
    Since relations $(1)-(3)$ each contain some $y_k$, it suffices to show that this term is nonzero modulo relation $(5)$. 

    Recall the coinvariant ring $H^{*}(\mathrm{Fl}(\C^n)) \cong  R = \C[x_1,\ldots, x_n]/I$ where $I$ is the ideal generated by the elementary symmetric functions $e_1,\ldots, e_n$. If the right-hand side of Equation \ref{eq:onlyxmono} is nonzero in $R$, then it is nonzero in $H^{*}(\Hess(S,h))$. The \textbf{Artin basis} of $R$ consists of monomials $\{x_1^{i_1}x_2^{i_2}\cdots x_n^{i_n}\,|\, 0\leq i_j\leq j-1\}$. Any permutation of the Artin basis $\{x_{\pi(1)}^{i_1}x_{\pi(2)}^{i_2}\cdots x_{\pi(n)}^{i_n}\,|\, 0\leq i_j\leq j-1\}_{\pi\in \Sn}$ also forms a basis of $R$.  

    Suppose that $$ x_n^{\ell_1}x_{n-1}^{\ell_2}\cdots x_{2}^{\ell_{n-1}}\left( \prod_{\ell=2}^{h(1)}(x_1-x_{\ell}) \right) = \sum c_k\, x_1^{k_1} x_2^{k_2} \cdots x_n^{k_n} $$ Let $m$ be the largest index such that the exponent on $x_m$ in $x_n^{\ell_1}x_{n-1}^{\ell_2}\cdots x_{2}^{\ell_{n-1}}$ is $0$, which exists since it does not contain a factor of $\prod_{\ell=h(1)+1}^{n} x_{\ell}$. In particular, $m\geq h(1)+1$. We show that each term of the sum above is an element of the permuted Artin basis for the permutation $\pi=(1m)$. 

    First, since $x_m$ does not divide $x_n^{\ell_1}x_{n-1}^{\ell_2}\cdots x_{2}^{\ell_{n-1}}$, and $m\geq h(1)+1$, we have that $k_m=0$. Next, for each $x_j$ with $2\leq j\leq h(1)$, the exponent $\ell_{n-j+1}$ satisfies $0\leq \ell_{n-j+1} \leq j-2$ by definition. Hence, the exponent $k_j$ satisfies $0\leq k_j \leq j-1$. For all $x_j$ with $h(1)+1\leq j\leq n$, other than $x_m$, we have $k_j = \ell_{n-j-1} \leq j-2$. Lastly, we have that the exponent on $x_1$ is $k_1 = h(1)-1 \leq m-2$. Thus each term of the sum is an element of this permuted Artin basis. 

    Therefore $x_n^{\ell_1}x_{n-1}^{\ell_2}\cdots x_{2}^{\ell_{n-1}}\left( \prod_{\ell=2}^{h(1)}(x_1-x_{\ell}) \right)$ is nonzero in $R$, and thus is also nonzero in the cohomology ring $H^{*}(\Hess(S,h))$. So, modulo the relations in the ideal $I$, the given set of elements of $H^{*}(\Hess(S,h))$ are linearly independent. Consequently, the set $\{ x_n^{\ell_1}x_{n-1}^{\ell_2}\cdots x_{2}^{\ell_{n-1}}\,y_k \}_{1\leq k\leq n}$ is linearly independent. 
\end{proof}

This new set behaves nicely with regard to the $\Sn$ action - since $\Sn$ fixes each of the $x_k$ variables, and permutes the $y_k$ variables in the natural way, this basis decomposes into a number of copies of the natural permutation representation of $\Sn$. 

\begin{corollary}
    The $\Sn$-module $H^*(\Hess(S,h))$ decomposes into $(n-h(1))(n-2)!$ copies of the natural permutation representation, and $n(h(1)-1)(n-2)!$ copies of the trivial representation. 
\end{corollary}

\begin{proof}
    Since $\{x_n^{\ell_1}x_{n-1}^{\ell_{2}}\cdots x_2^{\ell_{n-1}}\,y_k\}_{1\leq k\leq n}$ forms a linearly independent set of polynomial elements of $H^{*}(\Hess(S,h))$, for a fixed list of exponents $(\ell_1,\ldots,\ell_{n-1})$, the submodule generated by these monomials yield a copy of the natural permutation representation. The number of such vectors $(\ell_1,\ldots, \ell_{n-1})$ such that $\ell_{i}=0$ for some $1\leq i\leq n-h(1)$, and $0\leq \ell_j \leq n-j-1$ is $(n-h(1))(n-2)!$. The other basis elements of $H^*(\Hess(S,h))$ are fixed, and there are $n! - n((n-h(1))(n-2)!) = n(h(1)-1)(n-2)!$ of these basis elements. 
\end{proof}

As discussed in Section \ref{sec:background}, the natural permutation representation is sent by the Frobenius character map to $h_{(n-1,1)}$, and the trivial representation is sent to $h_{(n)}$. So, applying the involution $\omega$, the corresponding chromatic symmetric function decomposes as $X_{G_h}(\xvars) = (n-h(1))(n-2)!\,e_{(n-1,1)} + n(h(1)-1)(n-2)!\,e_{(n)}$.  

\section{Tableaux Bijections}\label{sec:bijections}

In this section, we examine bijections between the monomial basis elements and sets of tableaux.

\subsection{The Regular Nilpotent Case}

We begin this section by looking at regular nilpotent Hessenberg varieties, rather than regular semisimple ones, since the cohomology rings in this case are understood for all Hessenberg functions. In this section, $N$ is a regular nilpotent matrix, so in Jordan canonical form, $N$ has a single Jordan block with eigenvalue $0$.

\begin{proposition}[\cite{HHMPT}, Corollary 7.3]\label{nilpotentbasis}
Let $N$ be a regular nilpotent matrix and let $h$ be any Hessenberg function of length $n$. Then the following set of monomials form an additive basis for $H^*(\textrm{Hess}(N,h))$: 
$$ \mathcal{N}_h := \left\{ x_1^{i_1}\cdots x_{n}^{i_n}\,|\,0\leq i_k \leq h(k)-k\textrm{ for } 1\leq k <n\right\}$$
\end{proposition}

These basis elements can be visualized by looking at the associated Dyck path for the Hessenberg vector. Given a Hessenberg vector $h$ of length $n$, construct a path $D_h$ in the $(x,y)$-plane from $(0,0)$ to $(n,n)$ using the steps $\langle 0,1\rangle$ and $\langle 1,0\rangle$ as follows: 
For each $1\leq k\leq n$, if $h(k)=j$, then include the step $\langle 1,0\rangle$ starting at the point $(k-1,j)$. Since each $k$ occurs only once, there will be $n$ steps of the form $\langle 1,0\rangle$, and exactly one way to fill in the rest of the path using $\langle 0,1\rangle$ steps. 

Given a Dyck path $D_h$ for a Hessenberg vector $h$, we visualize the basis elements as follows: For each $k$, let $b_k$ be the number of full boxes below the path and above the main diagonal. Then $b_k=h(k)-k$, so $b_k$ is the largest power allowed on the $x_k$ in a basis element.

Regular nilpotent Hessenberg varieties are useful in the study of regular semisimple Hessenberg varieties in the following way: When $S$ is regular semisimple, $\Sn$ acts on $H^{*}(\Hess(S,h))$ by Tymoczko's dot action, and the fixed points of this action return $H^{*}(\Hess(N,h))$.

\begin{proposition}[\cite{AbHaHoMa}, Theorem B]\label{prop:regular_to_nilpotent_iso}
    Let $N$ be a regular nilpotent matrix, $S$ be a regular semisimple matrix, and let $h$ be any Hessenberg function of length $n$. Then there exists an isomorphism of graded algebras $$ \mathcal{A}: H^{*}(\Hess(N,h)) \to H^{*}(\Hess(S,h))^{\Sn}$$ where $H^{*}(\Hess(S,h))^{\Sn}$ is the set of $\Sn$-fixed points of $H^{*}(\Hess(S,h))$.
\end{proposition}

We can translate the action of $\Sn$ on $H^{*}(\Hess(S,h))$ to a trivial action on $H^{*}(\Hess(N,h))$. Since the action is trivial, the corresponding representation of $\Sn$ decomposes into the direct sum of one-dimensional trivial representations. Under the graded Frobenius map, these map to $q^d s_{(n)}$, for $d$ the degree of the grading on each subspace. From Equation (\ref{eq:HessPtabConnection}), we should be able to find a weight-preserving bijection between the set of monomial basis elements of $\mathcal{N}_h$ to the set of $P_h$-tableaux of shape $(n)' = (1^n)$.  

\begin{definition}
Define $\mathrm{PT}(h,\lambda)$ to be the set of $P_h$-tableaux of shape $\lambda$.     
\end{definition}

Let $\lambda=(1^n)$. We define a map $\varphi: \mathcal{N}_h \to \mathrm{PT}(h,\lambda)$ as follows.

\begin{definition}
Let $x_1^{i_1}\cdots x_n^{i_n}\in \mathcal{N}_h$. 
Begin with a $P_h$-tableau $T$ of a single box whose entry is $n$. For each $k=n-1,\ldots 1$, if $i_k = 0$, insert $k$ into a new box at the bottom of $T$, so $k$ occurs directly below some $\ell>k$. If $i_k>0$, then since $i_k\leq h(k)-k$, we have $k<h(k)$. List the entries $k+1,\ldots ,h(k)$, which already exist in $T$, in order from the lowest to highest row position in $T$. Insert $k$ in a new box directly above the $i_k$-th lowest entry of this list. Define $\varphi(x_1^{i_1}\cdots x_n^{i_n})$ to be the resulting tableau from this process.
\end{definition}

\begin{example}
    Let $h=(2,3,5,5,5)$, and consider the monomial $x_1x_3x_4\in \mathcal{N}_h$. We construct $\varphi(x_1x_3x_4)$ as follows.

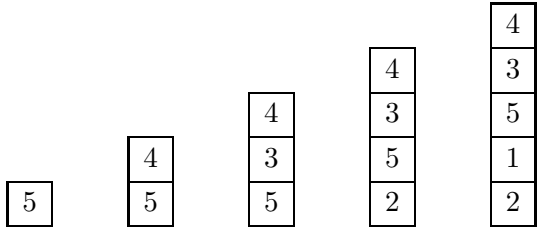
\begin{equation*}
    \begin{ytableau} \none \\ \none \\ \none \\ \none \\ 5 \end{ytableau} \hspace{1cm}
    \begin{ytableau} \none \\ \none \\ \none \\ 4 \\ 5 \end{ytableau} \hspace{1cm}
    \begin{ytableau} \none \\ \none \\ 4 \\ 3 \\ 5 \end{ytableau} \hspace{1cm}
    \begin{ytableau} \none \\ 4 \\ 3 \\ 5 \\ 2 \end{ytableau} \hspace{1cm}
    \begin{ytableau} 4 \\ 3 \\ 5 \\ 1 \\ 2 \end{ytableau} \hspace{1cm}
\end{equation*}
We start with a single box containing a $5$. Then, to insert $4$ with $i_4=1$ $P_h$-inversion, we insert the $4$ above the $5$. To insert $3$ with $i_3=1$ $P_h$-inversion, we insert the $3$ above the $5$ but below the $4$. Similarly insert a $2$ with no $P_h$-inversions, and a $1$ with one $P_h$-inversion. Notice that at each step, the number of elements in $P_h$ greater than $k$ that are incomparable to $k$ is $h(k)-k$, which is also the largest possible power $i_k$. 

\end{example}

\begin{theoremD}
The map $\varphi$ is a well-defined, weight-preserving bijection. 
\end{theoremD}

\begin{proof}
First, at each insertion step, if $i_k=0$, then $k$ is inserted at the bottom of the tableau, so the entry immediately above $k$ is larger, so adjacent entries are still $P_h$-non-decreasing. If $i_k>0$, by the definition of $P_h$, $k$ is incomparable to any $k<\ell \leq h(k)$, so adding $k$ above one of these $\ell$ maintains that the column of $T$ is $P_h$-non-decreasing. Hence after each insertion, $T$ is still a $P_h$-tableau, and so $\varphi$ is a well-defined function.

Now we show that $\varphi$ is injective: Consider a function $\psi: PT(h,\lambda) \to \mathcal{N}_h$ defined as follows: Given a $P_h$-tableau $T$ of shape $(1^n)$, let $\psi(T) = x_1^{i_1}\cdots x_n^{i_n}$ where $i_k$ is the number of inversions of $T$ with $k$ as the smaller entry. For each $k = 1,\ldots, n$, note that $k$ can form an inversion as the smaller entry with at most $h(k)-k$ entries: $k+1,\ldots, h(k)$. Hence each $i_k$ is between $0$ and $h(k)-k$, so $\psi(T)$ is in $\mathcal{N}_h$, and thus $\psi$ is well-defined. Further, given a monomial $x_1^{i_1}\cdots x_n^{i_n} \in \mathcal{N}_h$, $\varphi(x_1^{i_1}\cdots x_n^{i_n})$ is constructed so that each $k$ is inserted to form $i_k$ inversions as the smaller entry, since each $k$ is inserted above $i_k$ incomparable elements greater than it. Thus for any $T\in PT(h,\lambda)$, we have $\varphi(\psi(T))=T$.

Next, given $x_1^{i_1}\cdots x_n^{i_n} \in \mathcal{N}_h$, we show by induction on $n$ that $\varphi(x_1^{i_1}\cdots x_n^{i_n})$ is the unique $P_h$-tableau of shape $(1^n)$ that maps to $x_1^{i_1}\cdots x_n^{i_n}$ under $\psi$. When $n=1$, the only monomial in $\mathcal{N}_h$ is $1$, and the only $P_h$-tableau consists of a single box containing a $1$. Given $x_1^{i_1}\cdots x_n^{i_n}\in \mathcal{N}_h$, by the induction hypothesis, there is a unique $T$ with entries $2,\ldots, n$ such that $\psi(T) = x_2^{i_2}\cdots x_n^{i_n}$. Now consider the choices for where the $1$ entry can be inserted. If the $1$ is inserted in the bottom row, then it creates no new inversions and $\psi(T) = x_1^0x_2^{i_2}\cdots x_n^{i_n}$. If the $1$ is inserted above any entry $\ell>h(1)$, then $1$ and $\ell$ are comparable in $P_h$, and so $T$ is no longer a valid $P_h$-tableau. If the $1$ is inserted above any entry $\ell$ with $2\leq \ell \leq h(1)$, then there the same number of entries from the list $2,\ldots, h(1)$ below the $1$ in $T$ as there are inversions with $1$ as the smaller entry. Hence, for any $i_1$, there is exactly one way of inserting $1$ into $T$ so that $\psi(T) = x_1^{i_1}\cdots x_n^{i_n}$. Thus $\varphi(x_1^{i_1}\cdots x_n^{i_n})$ is the unique tableau with $\psi(\varphi(x_1^{i_1}\cdots x_n^{i_n})) = x_1^{i_1}\cdots x_n^{i_n}$, and therefore $\varphi$ is a bijection.
   
Lastly, we note that for any $x_1^{i_1}\cdots x_n^{i_n}$, by construction, $T = \varphi(x_1^{i_1}\cdots x_n^{i_n})$ has $i_k$ inversions with $k$ as the smaller entry for all $k$, so $\deg(x_1^{i_1}\cdots x_n^{i_n}) = \inv(T) = i_1+\cdots +i_n$.
\end{proof}


\subsection{The Regular Semisimple Case}

Now we turn our attention to regular semisimple Hessenberg varieties. Recall the partial set of basis elements $B_1$ of the cohomology ring $H^{*}(\Hess(S,h))$ when $h=(h(1),n,\ldots, n)$:
$$ B_1 = \left\{ x_1^{i_1} x_2^{i_2} \cdots x_n^{i_n} \textrm{ without a factor of } \prod_{\ell=1}^{h(1)} x_{\ell} \,\Bigg|\, 0\leq i_j \leq n-j\right\} $$

Let $\lambda = (1^n)$. We will define a map $\varphi$ between $B_1$ and $\mathrm{PT}(h,\lambda)$, and the corresponding inverse map $\psi$. 

\begin{definition}
Let $x_1^{i_1}\cdots x_n^{i_n}\in B_1$. Begin with a $P_h$-tableau $T$ of a single box whose entry is $n$. For each $k=n-1,\ldots, 1$, insert $k$ into $T$ above exactly $i_k$ of the existing entries. 

Let $k'$ be the smallest index in $1,\ldots, h(1)$ such that $i_{k'}=0$, which exists by the definition of $B_1$. By the above construction, after inserting $1$ through $n$, $k'$ will be on the bottom of $T$. If $k'=1$, then set $\varphi(x_1^{i_1}\cdots x_n^{i_n})$ to be $T$. If $1<k'\leq h(1)$, then slide the entry $k'$ up until it is directly below the $1$, and define $\varphi(x_1^{i_1}\cdots x_n^{i_n})$ to be $T$ after this slide. 
\end{definition}

We now define the map $\psi$ from $\mathrm{PT}(h,\lambda)$ to $B_1$. Let $P_{N}$ be the poset corresponding to $h=(n,\ldots, n)$, in which all elements are incomparable.

\begin{definition}
Let $T\in \mathrm{PT}(h,\lambda)$. If $1$ is in the bottom row or the second row, let $\hat{T}=T$. Otherwise, construct $\hat{T}$ by shifting the entry directly below the $1$ to be in the bottom row of $T$. Define $\psi(T)$ to be $x_1^{i_1}\cdots x_n^{i_n}$ where $i_k$ is the number of $P_{N}$-inversions in $\hat{T}$ with $k$ as the smaller entry. 
\end{definition}

\begin{example}
Let $h=(3,5,5,5,5)$, and consider the monomial $x_1^2 x_3x_4 \in B_1$. We construct $\varphi(x_1^2 x_3x_4)$ as follows. 

    
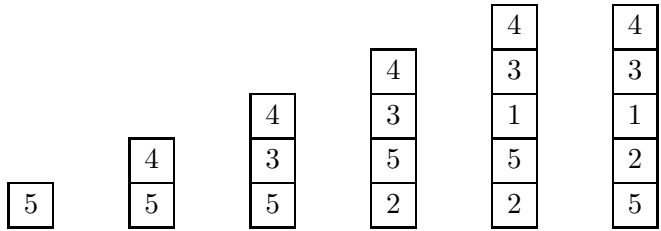
\begin{equation*}
        \begin{ytableau} \none \\ \none \\ \none \\ \none \\ 5 \end{ytableau} \hspace{1cm}
        \begin{ytableau} \none \\ \none \\ \none \\ 4 \\ 5 \end{ytableau} \hspace{1cm}
        \begin{ytableau} \none \\ \none \\ 4 \\ 3 \\ 5 \end{ytableau} \hspace{1cm}
        \begin{ytableau} \none \\ 4 \\ 3 \\ 5 \\ 2 \end{ytableau} \hspace{1cm}
        \begin{ytableau} 4 \\ 3 \\ 1 \\ 5 \\ 2 \end{ytableau} \hspace{1cm}
        \begin{ytableau} 4 \\ 3 \\ 1 \\ 2 \\ 5 \end{ytableau} 
    \end{equation*}
We start with a single box containing a $5$. Then we insert the $4$ above one existing entry, the $3$ above one existing entry, the $2$ above no existing entries, and the $1$ above two existing entries. Since $1<_{P_h}5$ are comparable, the resulting tableau is not a $P_h$-tableau, so we shift the $2$ (which is incomparable to the $1$) to be directly below the $1$. In the second-to-last tableau, the number of $P_{N}$-inversions with each $k$ as the smaller entry is exactly $i_k$, so reading these inversions returns the monomial $x_1^{2}x_3x_4$.
\end{example}

\begin{theoremB}
The map $\varphi$ is a well-defined bijection. 
\end{theoremB}

\begin{proof}
Let $x_1^{i_1}\cdots x_n^{i_n} \in B_1$. First we show that $\varphi(x_1^{i_1}\cdots x_n^{i_n})$ is a well-defined $P_h$-tableau. Since $h=(h(1),n,\ldots, n)$, the entries $2,\ldots, n$ are incomparable in $P_h$. So for $k=n-1,\ldots, 2$, inserting $k$ directly above $i_k$ existing entries in $T$ creates a $P_h$-tableau with exactly $i_k$ $P_h$-inversions where $k$ is the smaller entry. Then, inserting a $1$ above $i_1$ entries creates a $P_{N}$-tableau (but not necessarily a $P_{h}$-tableau) $T$ where the $1$ forms $i_1$ $P_{N}$-inversions. If $i_1=0$, then $\varphi(x_1^{i_1}\cdots x_n^{i_n}) = T$ is a $P_{h}$-tableau, since $1$ is not directly above anything comparable to it. 

If $i_1\neq 0$, then since $x_1^{i_1}\cdots x_n^{i_n}$ is in $B_1$, there exists some smallest index $k' \in \{2,\ldots, h(1)\}$ such that $i_{k'}=0$. By construction, $k'$ will be in the bottom row of $T$, so $k'$ forms no $P_h$-inversions as the smaller entry. Then, shifting $k'$ to be directly below the $1$ in $T$ creates a $P_h$-tableau, since $1$ is incomparable in $P_h$ to $k'$, and $k'$ is incomparable in $P_h$ to all other entries. Hence $\varphi(x_1^{i_1}\cdots x_n^{i_n})$ is a $P_h$-tableau, so $\varphi$ is well-defined. 

Now, let $T\in \mathrm{PT}(h,\lambda)$. If $1$ is in the bottom row, then it forms no $P_N$-inversions as the smaller entry, so the exponent on $x_1$ in $\psi(T)$ is $0$. When $1$ is in a higher row, we construct $T$ by shifting the entry $k'$ that is below the $1$ to the bottom of the tableau. So $k'$ forms no $P_N$-inversions as the smaller entry, and in $\psi(T)$, the exponent on $x_{k'}$ is zero. Since $T$ is a $P_h$-tableau, we have that $1\leq k'\leq h(1)$, so the monomial $\prod_{\ell=1}^{h(1)}x_{\ell}$ does not divide $\psi(T)$. Then, after the shift, $T$ is a $P_N$ tableau, so for all $j\in [n]$, $j$ can form a $P_N$-inversion as the smaller entry with at most $n-j$ other entries. Thus the exponent $i_j$ on $x_j$ satisfies $0\leq i_j\leq n-j$. Therefore $\psi(T)$ is in $B_1$, so $\psi$ is well-defined. 

Next, given $x_1^{i_1}\cdots x_n^{i_n}\in B_1$, we show that $\psi(\varphi(x_1^{i_1}\cdots x_n^{i_n}))=x_1^{i_1}\cdots x_n^{i_n}$. In the construction of $\varphi(x_1^{i_1}\cdots x_n^{i_n}),$ we constructed the intermediate $P_{N}$-tableau $T$ such that each $k$ in $1,\ldots, n$ has exactly $i_k$ $P_{N}$-inversions, before shifting the entry $k'$ from the bottom row to directly below the $1$. The map $\psi$ shifts the $k'$ back to the bottom row, and then counts the $P_{N}$-inversions with each $k$ as the smaller entry, which is how $T$ is originally constructed. Hence $\psi(\varphi(x_1^{i_1}\cdots x_n^{i_n}))= x_1^{i_1}\cdots x_n^{i_n}$. 

Further, for all $j = 1,\ldots, k'-1$, since $k'$ is below $j$ in $T$, we have that $(j,k')$ is an inversion with $j$ as the smaller entry, so the exponent on $x_j$ in $\psi(T)$ is nonzero. So $x_{k'}$ is the variable with smallest index that has an exponent of $0$. Therefore, in the construction of the intermediate tableau for $\varphi(\psi(T))$, $k'$ is in the bottom row, and so $k'$ is the entry that is shifted below the $1$. Further, by construction, the intermediate tableau has $i_k$ inversions with $k$ as the smaller entry, which is also true of $T$. Therefore $\varphi(\psi(T))=T$, and so we have that $\varphi$ is a bijection. 
\end{proof}

As with the nilpotent case, we have a bijection between the $\Sn$-fixed points of $H^{*}(\Hess(S,h))$ and $P_h$-tableaux of shape $(1^n)$ when $h=(h(1),n,\ldots, n)$. We now construct a bijection for the other monomials in $H^{*}(\Hess(S,h))$, given by the set $B_3$:
\begin{equation*} B_3 = \left\{ x_n^{\ell_1}\cdots x_{2}^{\ell_{n-1}}(y_{k+1}-y_1) \textrm{ with no factor of } \prod_{\ell=h(1)+1}^{n}x_{\ell} \,{\Bigg|}\, 0\leq \ell_j\leq n-j-1,\, 1\leq k\leq n-1\right\} \end{equation*}
\begin{definition}
    Define $\PSPT(h,\lambda)$ to be the set of pairs $(S,T)$, where $S$ is a standard Young tableau of shape $\lambda$ and $T$ is a $P_h$-tableau of shape $\lambda$. 
\end{definition}

Let $\mu = (2,1^{n-2})$ and $h=(h(1),n\ldots, n)$. We define a map $\varphi: B_3 \to \PSPT(h,\mu)$ as follows.

\begin{definition}
Let $x_n^{\ell_1}\cdots x_{2}^{\ell_{n-1}}(y_{k}-y_1)\in B_3$. First, let $S$ be the unique standard tableau of shape $\mu$ with entries $1$ and $k$ in the bottom row. Next, let $j$ be the largest entry among $h(1)+1, \ldots, n$ such that the exponent $\ell_{n-j+1}$ on $x_j$ is zero. Begin with a tableau $T$ with a single row of two boxes, containing a $1$ and $j$. Then for each $i\geq 2$, other than $j$, insert $i$ into the left column so that it is under exactly $(i-2)-\ell_{n-i+1}$ of the current entries in the left column. Define $\varphi(x_n^{\ell_1}\cdots x_{2}^{\ell_{n-1}}(y_{k}-y_1))$ to be the pair $(S,T)$ resulting from this construction.  
\end{definition}

We now define the inverse map $\psi:\PSPT(h,\mu)\to B_3$. 

\begin{definition}
    Let $I_h$ be the set of potential $P_h$-inversions in a $P_h$-tableau $T$, excluding those containing a $1$, so $I_h = \{(i,j)\,|\,1<i<j\textrm{ and }j\leq h(i)\}$. When $h=(h(1),n,\ldots, n)$, we get $I_h =\{(i,j)\,|\,1<i<j\}$. Let $(S,T) \in \PSPT(h,\mu)$. Let $1$ and $k$ be the entries in the bottom row of $S$. Let $\ell_{n-j+1}$ be the number of $P_h$-inversions in $I_h\setminus \Inv_h(T)$ with $j$ as the larger entry. Define $\psi(S,T)$ to be the monomial $x_n^{\ell_1}\cdots x_2^{\ell_{n-1}}(y_k-y_1)$ with $k$ and $\ell_{n-j+1}$ defined above. 
\end{definition}

\begin{example}
    Let $h=(3,5,5,5,5)$, and consider the monomial $x_5^2x_3(y_2-y_1) \in B_3$. Note that $j=4$ is the largest index where $x_j$ has an exponent of zero. We construct $\varphi(x_5^2x_3(y_2-y_1))$ as follows.
    \begin{equation*}
        S \,=\, \begin{ytableau} 5 \\ 4 \\ 3 \\ 1 & 2 \end{ytableau} \hspace{1cm} \hspace{1cm}
        \begin{ytableau}\none \\ \none \\ \none \\ 1 & 4 \end{ytableau} \hspace{1cm}
        \begin{ytableau}\none \\ \none \\ 2 \\ 1 & 4 \end{ytableau} \hspace{1cm}
        \begin{ytableau}\none \\ 3 \\ 2 \\ 1 & 4 \end{ytableau} \hspace{1cm}
        \begin{ytableau} 3 \\ 5 \\ 2 \\ 1 & 4 \end{ytableau} \,=\,T
\end{equation*}
$S$ is defined to be the unique standard Young tableaux of shape $(2,1^{n-2})$ with a $1$ and $2$ in the bottom row. Then, to form $T$, we start with a single row containing a $1$ and a $4$. We then insert a $2$ in the left column underneath $(2-2)-0 = 0$ entries, a $3$ in the left column underneath $(3-2)-1 = 0$ entries, and a $5$ in the left column underneath $(5-2)-2=1$ entry. 
    
To find the image under $\psi$ of this pair, the exponent $\ell_{n-j+1}$ on $x_j$ is the number of potential $P_h$-inversions $(i,j)$ with $i\neq 1$ that do not appear in $T$. The $2$ has no possible $P_h$-inversions as the larger entry, so $\ell_4=0$, but the $3$ is missing the inversion with $2$, so $\ell_{3}=1$. Since the $4$ is in the bottom row, it forms all possible $P_h$-inversions as the larger entry, so $\ell_2=0$. Then the $5$ forms an inversion with $3$, but not $2$ or $4$, so $\ell_1=2$. Hence $\psi(S,T)=x_5^{2}x_3(y_2-y_1)$.  
\end{example}

\begin{theoremC}
The map $\varphi$ is a well-defined bijection.
\end{theoremC}

\begin{proof}
Let $x_n^{\ell_1}\cdots x_2^{\ell_{n-1}}(y_k-y_1)\in B_3$ and suppose that $\varphi(x_n^{\ell_1}\cdots x_2^{\ell_{n-1}}(y_k-y_1)) = (S,T)$. By construction, $S$ is a standard tableau. Since $x_{h(1)+1}\cdots x_n$ does not divide the given monomial, there is some largest $j$ among $h(1)+1,\ldots, n$ such that the exponent $\ell_{n-j+1}$ on $x_j$ is zero. In the construction of $T$, the initial tableau has a row containing a $1$ and $j$. Since $j\geq h(1)+1$, we have that $1\to j$ forms a chain in $P_h$, so this row forms a valid $P_h$-tableau. Note that, by the definition of $B_3$, for each $i$, $0\leq \ell_{n-i+1}\leq i-2$. Then, at each insertion of $i\geq 2$, $i$ is inserted underneath at most $i-2$ entries in the column. If $i<j$, then $i$ will not be inserted beneath the $1$ in the bottom row. If $i>j$, then by the definition of $j$, $\ell_{n-j+1}\geq 1$, so $i$ is inserted underneath at most $i-3$ entries in the left column, and again will not be inserted under the $1$ in the bottom row. Since all of $2,\ldots, n$ are incomparable in $P_h$, after each $i$ is inserted, $T$ will remain a $P_h$ tableau. Hence the map $\varphi$ is well-defined. 

Now suppose $T$ is a $P_h$ tableau of shape $\mu$. For each $x_j$, the number of $P_h$-inversions $(i,j)$ of $T$ with $1<i<j$ is at most $j-2$, so the exponent $\ell_{j}$ on $x_{n-j+1}$ is at most $(n-j+1)-2 = n-j-1$. Next, by the definition of $P_h$-tableau, the bottom row of $T$ is a chain in $P_h$, so it has entries $1$ and $m$ for some $m>h(1)$. Then all $P_h$-inversions $(i,m)$ with $1<i<m$ are present in $T$, so the exponent $\ell_{n-m+1}$ on $x_m$ is $0$. Hence $\prod_{\ell=h(1)+1}^{n}x_{\ell}$ does not divide $\psi(T,S)$. Finally, if $S$ is a standard tableau, then the bottom row must contain entries $1$ and $k$ for some $2\leq k\leq n$. Hence $\psi(T,S) \in B_3$, and so $\psi$ is well-defined.

We now show that $\varphi$ and $\psi$ are inverses. For $x_n^{\ell_1}\cdots x_2^{\ell_{n-1}}(y_k-y_1)\in B_3$, note that the pair $\varphi(x_n^{\ell_1}\cdots x_2^{\ell_{n-1}}(y_k-y_1)) = (S,T)$ is constructed so that each entry $i=2,\ldots,n$ in $T$ (excluding $j>1$ in the bottom row) forms $i-2 - \ell_{n-i+1}$ $P_h$-inversions as the larger value. So there are $\ell_{n-i+1}$ $P_h$-inversions in $I_h\setminus \Inv(T)$ with $i$ as the larger entry. Further, $S$ is constructed so that the bottom row of $S$ corresponds to the factor $(y_k-y_1)$. Thus $\psi(\varphi(x_n^{\ell_1}\cdots x_2^{\ell_{n-1}}(y_k-y_1))) = x_n^{\ell_1}\cdots x_2^{\ell_{n-1}}(y_k-y_1)$. 

Next, note that, at each step of the construction of $T$ from $x_n^{\ell_1}\cdots x_2^{\ell_{n-1}}(y_k-y_1)$, since $2,\ldots, n$ are all incomparable in $P_h$, inserting an entry $i>1$ into the left column below $k$ existing entries gives $k$ $P_h$-inversions with $i$ as the larger entry. Hence there is a unique place to insert $i$ so that $i$ forms $(i-2)-\ell_{n-i+1}$ $P_h$-inversions as the larger entry. Thus $\varphi(x_n^{\ell_1}\cdots x_2^{\ell_{n-1}}(y_k-y_1))$ is the unique pair of tableaux such that $\psi$ maps the pair to $x_n^{\ell_1}\cdots x_2^{\ell_{n-1}}(y_k-y_1)$. Therefore $\psi$ is a left-inverse of $\varphi$, and so $\varphi$ is a bijection.
\end{proof}

\section{New Proof of the Poincar\'{e} Polynomial}\label{sec:poincarepolynomial}

Recall that, given a graded vector space $V$ over a field $k$, if $V = \bigoplus_{i\in \N} V_i$ with each subspace $V_i$ consisting of vectors of degree $i$ being finite dimensional, then the Poincar\'{e} polynomial of $V$ is $$\mathrm{Poin}(V,q) = \sum_{i\in \N} \dim_k(V_i) q^{i} $$

In Equation \ref{eq:HessPtabConnection}, we saw that we could express the image of regular semisimple Hessenberg varieties under the graded Frobenius map in terms of Schur functions in the following way, due to the chromatic quasisymmetric function defined by Shareshian and Wachs: 
\begin{equation} \sum_{j=0}^{|E|} \mathrm{Frob}(H^{2j}(\Hess(S,h)))q^j = \omega X_G(\mathbf{x};q) = \sum_{\lambda\vdash n}\left( \sum_{T\in PT(\lambda)} q^{\inv_h(T)}\right) s_{\lambda'} \end{equation}

So, we can write the Poincar\'{e} polynomial of a regular semisimple Hessenberg variety as follows:
$$ \mathrm{Poin}(\Hess(S,h),q) = \sum_{\lambda\vdash n} \left( \sum_{T\in \mathrm{PT}(h,\lambda)} q^{\inv_h(T)}\right) \#\mathrm{SYT}(\lambda) $$ 
since the dimension of the irreducible representation $V_\lambda$ is the number of standard Young tableaux of shape $\lambda$. We use this formula to provide an alternate proof of the formula of the Poincar\'{e} polynomial of the regular semisimple Hessenberg variety $\Hess(S,h)$ when $h=(h(1),n,\ldots, n)$, originally proven by Abe, Horiguchi, and Masuda in \cite{AHM}.

\begin{theorem}[\cite{AHM}, Lemma 3.2]
If $h=(h(1),n,\ldots, n)$, then the Poincar\'{e} polynomial of $\Hess(S,h)$ is given by $$ \mathrm{Poin}(\Hess(S,h),q) = \frac{1-q^{h(1)}}{1-q}\,\prod_{j=1}^{n-1}\frac{1-q^j}{1-q} + (n-1)q^{h(1)-1}\frac{1-q^{n-h(1)}}{1-q}\,\prod_{j=1}^{n-2}\frac{1-q^j}{1-q} $$
\end{theorem}

Writing this in terms of $q$-analogues, we get the following expression: 
$$ \mathrm{Poin}(\Hess(S,h),q) = h(1)_q(n-1)_q! + (n-1)q^{h(1)-1}(n-h(1))_q(n-2)_q!$$

\begin{proof}
From above, we know that $$ \mathrm{Poin}(\Hess(S,h),q) = \sum_{\lambda\vdash n} \left( \sum_{T\in \mathrm{PT}(h,\lambda)} q^{\inv_h(T)}\right) \#\mathrm{SYT}(\lambda). $$
  
Let $h=(h(1),n,\ldots, n)$. All chains in $P_h$ have length two and include the element $1$. Since distinct rows in a $P_h$ tableaux need to contain entries from distinct chains in $P_h$, the only shapes $\lambda$ with a nonzero number of $P_h$-tableaux are $\lambda = (1^n)$ and $\mu = (2,1^{n-2})$. Further, we have that $\#\mathrm{SYT}(\lambda) = 1$ and $\#\mathrm{SYT}(\mu)=n-1$. 

For $\lambda=(1^n)$, we need to count the $P_h$-inversions in the $P_h$ tableaux of this shape. Since the element $1$ is incomparable to each of $2,\ldots, h(1)$, it can form between $0$ and $h(1)-1$ inversions as the smaller entry. For each $i=2,\ldots, n$, the entry $i$ can form up to $n-i$ inversions as the smaller entry. Hence, we get that $$ \sum_{T\in \mathrm{PT}(h,\lambda)}q^{\inv_h(T)} = (1+q+\cdots+q^{h(1)-1})(1+q+\cdots+q^{n-2})! = h(1)_q(n-1)_q!.$$

For $\mu=(2,1^{n-2})$, the bottom row of any $P_h$-tableaux of shape $\mu$ must be filled with entries from a chain in $P_h$, so it contains a $1$ and an $i$ for some $i=h(1)+1,\ldots, n$. Then, since $i>1$, it is incomparable with all other $j\neq 1$, so the entry $i$ in the bottom row forms inversions as the larger entry with the entries $2,\ldots, i-1$, of which there are $i-2$. So this entry contributes between $h(1)-1$ and $n-2$ inversions to the $P_h$-tableaux as the larger entry. Then, for the column entries of $j=2,\ldots, n$ and $j\neq i$, if $j<i$, then $j$ forms an inversion with $i$ where $j$ is the smaller entry (which was already counted), and can form an inversion as the smaller entry with the other $n-j-1$ entries larger than $j$. If $j>i$, then $j$ does not form an inversion with $i$, and can form an inversion as the smaller entry with any of the $n-j$ entries larger than $j$. In each case, there is a unique placement for $j$ giving each set of inversions. Hence we have $$ \sum_{T\in \mathrm{PT}(h,\mu)} q^{\inv_h(T)} = (q^{h(1)-1}+q^{h(1)}+\cdots+q^{n-1}) (1+q+\cdots+q^{n-3})! = q^{h(1)-1}(n-h(1))_q (n-2)_q! $$

Therefore, for $\lambda=(1^n)$ and $\mu=(2,1^{n-2})$, we have the Poincar\'{e} polynomial of $\Hess(S,h)$ as follows:
\begin{align*} \mathrm{Poin}(\Hess(S,h),q) &= \sum_{T\in \mathrm{PT}(h,\lambda)} q^{\inv_h(T)} + (n-1)\sum_{T\in \mathrm{PT}(h,\mu)} q^{\inv_h(T)} \\ &= h(1)_q (n-1)_q! + (n-1)q^{h(1)-1}(n-h(1))_q(n-2)_q! \end{align*}
This completes the proof. \end{proof}

\printbibliography

\end{document}